\documentclass[journal]{IEEEtran}

\usepackage[utf8]{inputenc}
\usepackage{amsmath,amsthm,amsfonts,amssymb,bm,dsfont}
\usepackage{graphicx,boxedminipage}
\let\labelindent\relax % Conserta erro de conflito de definição com o IEEEtran
\usepackage{enumitem}
\usepackage[0]{editing}
\usepackage{cite}
\usepackage{icomma}
\usepackage{xcolor}

\DeclareMathOperator{\trace}{Tr}
\DeclareMathOperator{\Tr}{Tr}
\DeclareMathOperator{\col}{col}
\DeclareMathOperator{\diag}{diag}

\DeclareMathOperator{\MSD}{MSD}
\DeclareMathOperator{\EMSE}{EMSE}
\DeclareMathOperator{\MSE}{MSE}

\newcommand{\norm}[1]{\ensuremath{\| #1 \|}}
\newcommand{\abs}[1]{\ensuremath{\left\vert #1 \right\vert}}
\newcommand{\eps}{\ensuremath{\epsilon}}

\newcommand{\kron}{\otimes}

\newcommand{\re}{\mathbb{R}}
\newtheorem{Theorem}{Theorem}%[section]

\theoremstyle{definition}
\newtheorem{defn}{Def.}%[section]

\DeclareMathOperator{\E}{E}
\newcommand{\tilw}{\ensuremath{\widetilde{w}}}
\newcommand{\tilmu}{\ensuremath{\widetilde{\mu}}}

\newcommand{\abar}{\ensuremath{\bar{a}}}

\newcommand{\pbar}{\ensuremath{\bar{p}}}
\newcommand{\Lambdabar}{\ensuremath{\bar{\Lambda}}}
\newcommand{\lambdabar}{\ensuremath{\bar{\lambda}}}
\newcommand{\Thetabar}{\ensuremath{\bar{\Theta}}}
\newcommand{\Psitl}{\ensuremath{\widetilde{\psi}}}
\newcommand{\calMbar}{\ensuremath{\overline{\mathcal{M}}}}

\newcommand{\calM}{\ensuremath{\mathcal{M}}}
\newcommand{\calN}{\ensuremath{\mathcal{N}}}

\newcommand{\calL}{\ensuremath{\mathcal{L}}}
\newcommand{\calQ}{\ensuremath{\mathcal{Q}}}
\newcommand{\colone}{\ensuremath{\mathds{1}}}
\newcommand{\eqdef}{\triangleq}

\begin{document}

\title{Distributed Universal Adaptive Networks}

%\Author{Cassio G. Lopes $^{1}$*, Vítor H. Nascimento $^{1}$ and Luiz F. O. Chamon $^{2}$}

\author{Cassio G. Lopes ,~\IEEEmembership{Senior Member,~IEEE,}
        Vítor H. Nascimento,~\IEEEmembership{Senior Member,~IEEE,}
        Luiz F. O. Chamon,~\IEEEmembership{Member,~IEEE}% <-this % stops a space
\thanks{C. G. Lopes and V. H. are with the Dept. of Electronics and Systems, Escola Politécnica, University of Sao Paulo, São Paulo, SP, Brasil, \{cassio.lopes,vitnasci\}@usp.br.\\ L. F. O. Chamon is with the Excellence Cluster for Simulation Technology of the University of Stuttgart, Stuttgart, Germany, luiz.chamon@simtech.uni-stuttgart.de.}% <-this % stops a space

\thanks{This research was funded by the ELIOT project,  São Paulo Research Foundation (FAPESP) 2018/12579-7 and ANR-18-CE40-0030. The work of L. F. O. Chamon is funded by the Deutsche Forschungsgemeinschaft (DFG, German Research Foundation) under Germany's Excellence Strategy (EXC 2075-390740016).}

}


% \markboth{Journal of \LaTeX\ Class Files,~Vol.~14, No.~8, August~2015}%
% {Shell \MakeLowercase{\textit{et al.}}: Bare Demo of IEEEtran.cls for IEEE Journals}

%
%

\maketitle

\begin{abstract}
Adaptive networks (ANs) are effective real time techniques to process and track events observed by sensor networks and, more recently, to equip Internet of Things (IoT) applications. \add{ANs operate over nodes equipped with collaborative adaptive filters that solve distributively an estimation problem common to the whole network. However, they \textit{do not} guarantee that nodes do not lose from cooperation, as compared to its non-cooperative operation; that poor nodes are rejected and exceptional nodes estimates reach the entire network; and that performance is uniform over all nodes. In order to enforce such properties, this work introduces the concept of distributed universal estimation, which encompasses the new concepts of local universality, global universality and universality with respect to the non-cooperative operation. We then construct a new cooperation protocol that is proven to be distributively universal, outperforming direct competitors from the literature, as shown by several simulations. Mean and mean-square analytical models are developed, with good agreement between theory and simulations.}

\end{abstract}



% \fi
\IEEEpeerreviewmaketitle

% \section{Introduction}

% for IEEE journal papers produced under \LaTeX\ using
% IEEEtran.cls version 1.8b and later.
% % You must have at least 2 lines in the paragraph with the drop letter
% % (should never be an issue)
% I wish you the best of success.

\vspace{-0.5cm}

\section{Introduction}
	\label{S:Intro}

\IEEEPARstart{I}{n} several applications, a network of interconnected agents is in charge of observing events in a field of interest, usually performing event-related tasks. Such events leave a space-time signature that may be registered by a number of sensors properly placed throughout the geographical area where the events take place \cite{Estrin2002}.

Applications \add{of this framework} are plentiful: detect \add{a signal of interest}; estimate physical quantities, such as temperature, pressure or wind velocity, solar incidence, the position and speed of a target, the spectrum of signals; \add{network synchronization}; identify structural failures, among others \cite{Cattivelli11d,Kulkarni2011,Yannick2012,Scaglione2007},  \cite{Lorenzo2013,Zhai2014,Scabin2016}. \add{In such applications, in the absence of a central node, an optimization problem common to the entire network must be solved distributively and cooperation among agents is either mandatory, or very desirable to promote improvement in network performance and robustness. The agents, or nodes, conduct partial processing over local data using low caliber processors and, via limited cooperation with nearby peers, the local results are collectively aggregated into a global solution that, ideally, should achieve the performance of an (hypothetical) central node omniscient of the network data. A subset of such problems is that of \textit{distributed estimation}, which is the scope of this work.}

\add{One of the strategies in the literature to promote a collective behavior towards an asymptotic global solution is the \textit{consensus strategy}, which usually consists of a distributed averaging of proper quantities related to the estimation problem, such as sample cross-covariance and auto-covariance, or even direct estimates \cite{Xiao2004, Scaglione2007, Scaglione2008, Schizas2009, Sandryhaila2014,kar_convergence_2011,Chen-PP-TSP2021}.
}

Most applications can be modeled as a vector of parameters to be estimated from the space-time data captured across the network. In this context, the concept of \textit{adaptive networks} (ANs) arose as an effective real time technique to estimate the application-related parameter vector \cite{Cassio-ICASSP07}, \cite{Cassio07i}, \cite{Cassio08d}, \cite{Cattivelli08d}. \add{When the comparison is meaningful, it has been shown that the cooperative strategy conducted by ANs may outperform that of consensus strategies  \cite{tu_diffusion_2012}. ANs are the focus of this work.}

In addition to data statistics, two factors define performance in ANs: the learning rules at the nodes and the cooperation protocol. The learning rule, such as the least mean-squares (LMS) or the recursive least-squares (RLS) rules \cite{Sayed08a,Diniz13a}, is selected according to performance requirements and the local available processor. The cooperation protocol must restrict communications to local interactions only, favoring energy savings. Typical cooperation protocols are the \add{incremental and the diffusion, with all their variants \cite{Cassio07i,Sayed2006, Cattivelli08d, Cassio2010-rand,Cassio08d,Plata2015,Cassio2008b,Chouvardas2011,Harrane2019,Jin2020,Jin2020b}.}

\add{In the incremental cooperation protocol, only one estimate is shared at a time and it is updated over a cycle visiting all the nodes exactly once \cite{Cassio07i,Sayed2006}. This technique allows for extreme energy savings, though it requires the definition of a Hamiltonian cycle across the network (an NP-hard problem) \cite{Bollobas1998} and the network becomes hostage to local node performance, which may vary greatly. Later, a randomized incremental version was proposed that avoided the Hamiltonian cycle and promoted, on average, a good degree of performance uniformity across the network \cite{Cassio2010-rand}.}

\add{The diffusion protocol \cite{Cassio08d,Cattivelli08d,Plata2015,Cassio2008b,Chouvardas2011,Arablouei2014,Sayin2014,Harrane2019,Jin2020,Jin2020b}} is usually preferred since it explores more efficiently the available information, \add{also providing a good performance uniformity across the network}, although at a higher energy cost, \add{as compared to the incremental protocols}. Whichever the diffusion type, at any node the protocol starts by fusing estimates retrieved from nearby nodes. The fusing is obtained via a local function, usually a fixed linear combination of the estimates. This fixed design follows specific rules, such as the uniform and the Metropolis,  which are node-degree-dependent convex combiners  \cite{Cassio08d}, and the relative-variance rule \cite{Tu11o}, which further accounts for node noise variance. A learning step follows, in which the fused estimate is injected into the local AF, which updates its estimate in response to the other node estimates and to its local data \cite{Cattivelli11d, Cassio08d, Cattivelli08d, Tu11o, Sayed14a}.

\add{Many works adopt the diffusion protocol and attempt to limit its cooperation complexity, saving energy and communication resources, while avoiding a corresponding network performance deterioration \cite{Cassio2008b}, \cite{Harrane2019, Arablouei2014, Sayin2014}. In \cite{Cassio2008b}, a probabilistic selection of a subset of neighbors dramatically decreased the local diffusion cooperation, thus with major energy savings, while maintaining the performance properties to a great extent. Other approaches propose reducing the cooperation load by transmitting compressed versions of the local estimates, as in \cite{Sayin2014}; or performing a double compression over the local learning and fusion steps before sharing quantities with other nodes \cite{Harrane2019}. In a similar vein, the scheme in \cite{Arablouei2014} proposes that each node transmits a subset of the local estimate entries to its neighbors at every iteration.}

\add{Another important line of work is comprised of selective cooperation policies over diffusion networks \cite{Cassio08d}, \cite{werner_distributed_2009}, \cite{Takahashi10d,Bes12n,Cassio2014,Jin2020,Jin2020b}.
The idea is not to blindly cooperate; instead, adopt rules that give emphasis to better nodes, or that improve estimation performance by, for example, transmitting only sufficiently novel information.}

\add{With the inception of ANs, it was rapidly noticed that adapting the local fusion combiners, assigning larger weights to select better nodes, without discarding the less-fortunate ones, yielded network performance improvement \cite{Cassio08d}, and different selective adaptive policies followed \cite{Takahashi10d,Bes12n,Cassio2014}. Also related are the works \cite{Jin2020,Jin2020b}, which elegantly adopt the classical universal adaptive convex  \cite{Jeronimo06m} and affine \cite{Candido2010} combinations at every node, but whose inputs are two independent and competing generic diffusion ANs; this generates an output AN whose \textit{average} network mean-square performance is guaranteed to be at least as good as the best average network performance between the two input competing ANs.
}

\add{Despite the improvement in average network mean-square  performance achieved by selective diffusion networks, some nodes may be better off working independently when their cooperative and non-cooperative performances are compared \cite{Cassio2014}. This raises a nontrivial question: when, or how, should a node cooperate or not cooperate?
In order to properly answer this question, in this work we depart from the standard standalone universal estimators \cite{Merhav98u,Singer99u}, adopted in \cite{Jin2020} and \cite{Jin2020b} at every node, and develop the concept of \textit{distributed universal networks}. In this new conceptual framework, an adaptive network will be considered universal if every node performs at least as well as the best available standalone node. This simple definition, due to the spatial data diversity, to the cooperation, and to the limited communication among nodes, unfolds into different kinds of universality: local universality, global universality, and universality with respect to (w.r.t.) the non-cooperative operation. These definitions are directly connected with desired properties of distributed adaptive systems, namely: (a) ability to reject bad nodes; (b) promotion of good nodes; (c) node performance homogeneity (See Section \ref{ssec:Defs_Universality}). Such properties are advocated here to assess what good performance means in the context of distributed adaptive estimation.
}

\add{After motivating proper definitions of distributed universality, a cooperation protocol is introduced that guarantees that nodes do not lose from cooperating: their performance will be always at least as good as if they operated individually, but often much better. The core idea is to preserve local estimates, while separately fusing the estimates received from neighboring nodes. A similar idea was presented in \cite{Bes12n}, in an effort to solely improve performance in heterogeneous networks, with a subsequent improved version in \cite{Bes2017}. However, the idea of recasting the distributed estimation problem in a new universal estimation framework was put forward in our preliminary work \cite{Cassio2014}, yielding a less complex and more efficient algorithm. Here, we extend that work in important ways: (a) formalize the aforementioned universality types (In \cite{Cassio2014} there were two types only); (b) prove that our adaptive distributed algorithm is indeed universal in the new sense; (c) develop analytic mean and mean-square models for our algorithm, for stationary and non-stationary cases; (d) provide comparisons with other relevant time-varying combiners from the literature \cite{Cassio08d,Takahashi10d,Bes12n,Bes2017}, in which the proposed algorithm stands out as the most efficient and the only one that is truly universal in all aspects.}

This paper is organized as follows. Section \ref{S:ANs} covers the fundamentals of adaptive networks, also introducing the main adaptive combiner strategies from the literature. \add{Section \ref{sec:D_U_AN} presents the original concept of universal estimation and how it should be upgraded to the distributed case, with a detailed discussion on the required new definitions of universality. In Section \ref{sec:U_protocol} the universal distributed algorithm is constructed. We prove that, under reasonable conditions, our algorithm is distributed universal.}
Section \ref{sec:Analysis} develops analytical models for the mean and mean-square algorithm evolution, also showing that the algorithm is stable and converges to the optimum vector in the mean, under typical conditions. The proposed strategy is then compared in Section \ref{S:Sims} to the two main competing algorithms from the literature, namely \cite{Takahashi10d} and \cite{Bes2017}, \add{also showing that the developed analytical mean-square model reasonably agrees with simulations.}  %% The template disregard the call for this particular section

\vspace{0.1cm}

\noindent \textit{Remarks on notation:} small font letters
refer to scalars and vectors, and capital letters to constants and matrices: $\epsilon$ is a scalar regularization factor; $M$ is the local filter order (a constant), and $A$ is the network constant adjacency matrix. We employ subscript indexing to denote time-varying vectors and matrices, and parentheses to describe time-varying scalars: at node $n$, $u_{n,i}$ is a (row) vector that collects the local scalar signal $u_n(i)$; $w_{n,i}$ is a local vector estimate for the network unknown vector $w^o$; and $H_{n,i}$ is a matrix that denotes the local learning rule at time $i$. This is usually clear from the context.

\vspace{-0.3cm}

\section{Adaptive networks with adaptive combiners}
	\label{S:ANs}

An adaptive network structure is modeled as an (un)directed graph \add{${\cal{G}}=(V,E)$}, where $V$ is the node set and $E$ is the edge set \cite{Bollobas1998}. Algebraically, it is convenient to represent the network by its adjacency matrix $A$, defined as $[A]_{n\ell}=1$ if nodes $n$ and $\ell$ are connected, and $[A]_{n\ell}=0$ otherwise. By definition a node is connected to itself, i.e., $[A]_{nn}=1$ for all $n\in V$. A connected network has a path connecting every two nodes $n$ and $\ell$. The neighborhood for node $n$ is the set ${\cal{N}}_n$
of nodes that have a direct connection with node $n$, including itself; that is, all the nodes that are at most one hop away from node $n$. The strict neighborhood ${\overline{\cal{N}}_n}$ of $n$ does not contain node $n$ itself; in other words, $\overline{\cal{N}}_n={\cal{N}}_n \backslash {n}$.

At time $i$, the $n$-th node has access to a scalar measurement $d_n(i)$ and to another signal $u_n(i)$, that is collected into an $1\times M$ local row regressor vector\footnote{The regressor may also have a more general structure. For instance, in adaptive antennas,  $u_{n,i}=[u_{n,1}(i)~~u_{n,2}(i)~~\dots ~~u_{n,M}(i)]$ \cite{Sayed08a}.}
\begin{equation}
    u_{n,i}\eqdef[u_n(i)~~u_n(i-1)~~\dots ~~u_n(i-M+1)]\text{.}\label{eq:tap.delay}
\end{equation}
The data model is then defined via a known local function $f_n$, subject to noise $v_n(i)$. Typically, $f_n$ is linear in terms of an unknown $M\times 1$ global vector of parameters $w^o$, i.e.,
\begin{equation}\label{eq:data_model}
d_n(i) = f_n[u_{n,i}] + v_n(i) = u_{n,i}w^o + v_n(i),
\end{equation}
\add{or by a linear-in-the-parameters nonlinear model such as a truncated Volterra series \cite{Nascimento2014}.}

The goal of the $N$-node network is to estimate the unknown vector $w^o$ from the available space-time data set $\{d_n(i),u_{n,i}\}$, $n=1,\ldots,N$.  Since all nodes have the common goal $w^o$, it makes sense to cooperate, which not only improves overall performance, but may also enforce stability over the distributed adaptive process running at the nodes. For that matter, node $n$ runs a local adaptive filter of the form
\begin{equation}\label{eq:AF}
    \psi_{n,i} = \psi_{n,i-1} + H_{n,i}u_{n,i}^T[d_n(i)-u_{n,i}\psi_{n,i-1}] \text{.}
\end{equation}
Equation (\ref{eq:AF}) represents a stand-alone AF running locally and returning at time $i$ an $M\times 1$ estimate $\psi_{n,i}$ for the unknown global vector $w^o$, where $H_{n,i}$ is an $M\times M$ positive definite matrix that defines the local adaptive rule. For a scalar step-size $\mu_n$, the most common choices are $H_{n,i}=\mu_n I$, for the LMS rule; and $H_{n,i}=\frac{\mu_n}{\|u_{n,i}\|^2+\epsilon} I$, for the normalized LMS (NLMS) rule, where $0<\epsilon\ll1$ is a small regularization factor \cite{Diniz13a}. Other rules are also possible.

Associated with node estimates are figures of merit that assess performance. They are inherited from standard adaptive filtering: the mean-square error (MSE), the excess mean-square error (EMSE), and the mean-square deviation (MSD). For node $n$, they are respectively defined as
\begin{equation}\label{eq:Local_Errors}
    \begin{aligned}
        \MSE_n(i) &~~\eqdef ~~E e^2_n(i) ~ = ~E[d_n(i)-u_{n,i}\psi_{n,i-1}]^2\\
        \EMSE_n(i) &~~\eqdef ~~E [u_{n,i}w^o-u_{n,i}\psi_{n,i-1}]^2\\
        \MSD_n(i) &~~\eqdef ~~E \|w^o-\psi_{n,i-1}\|^2.
    \end{aligned}
\end{equation}
The error definitions above are \textit{local quantities} and depend on \textit{which estimate is used locally} at the nodes to fulfill their tasks. In \eqref{eq:Local_Errors} it is assumed that the estimate $\psi_{n,i-1}$ is used locally, which is compatible to the case when the AFs evolve independently from other nodes. However, when cooperation takes place, additional definitions are needed (see below).

Cooperation may be achieved by fusing nearby estimates $\{\psi_{\ell,i}~,~\ell\in\calN_n\}$, in terms of local scalar combiners $c_{n\ell}$ to be designed. The resulting fused estimate $\phi_{n,i}$ is then injected into the learning step, i.e., into the local AF. Collecting the fusion and the learning steps results in the standard Diffusion LMS \cite{Cassio-ICASSP07,Cassio08d}, given at node $n$ by:
\begin{align}
    \phi_{n,i-1} &= \sum_{\ell \in {\cal{N}}_n} c_{n\ell} \psi_{\ell,i-1},    \label{eq:fusion}
	\\
	\psi_{n,i} &= \phi_{n,i-1} + H_{n,i}u_{n,i}^T[d_n(i)-u_{n,i}\phi_{n,i-1}]~\text{.}\label{eq:adapt}
\end{align}
Notice that the fusion step (\ref{eq:fusion}) aggregates space-time information from the neighborhood and tends to be a (much) better estimate for $w^o$ than $\psi_{n,i-1}$ in (\ref{eq:AF}). Subsequently, $\phi_{n,i-1}$ is used as an initial condition at time $i$ in the learning step \eqref{eq:adapt}, so that the local AF responds not only to its previous local estimate $\psi_{n,i-1}$, but also to those of its neighbors\footnote{Note that the error definitions \eqref{eq:Local_Errors} can be given in terms of $\psi_{n,i-1}$ or $\phi_{n,i-1}$, giving rise to the two versions of the diffusion protocol: respectively, combine-then-adapt (CTA) or adapt-then-combine (ATC) \cite{Sayed14a}.}. This is at the heart of the concept of an adaptive network: although the local node processes only local data, the fusion step couples local learning with nearby nodes. As every node $n$ proceeds the same way, the entire network adapts in real-time in order to track $w^o$ cooperatively, and in a fully distributed manner, from the observed space-time data $\{d_n(i), u_{n,i}\}$.

The aggregate estimate $\phi_{n,i-1}$ in (\ref{eq:fusion}) can be interpreted as a weighted least-squares estimate of $w^o$ given the received estimates $\{\psi_{\ell,i-1}\}$. This implies that the set $\{c_{n\ell}\ge 0\}$ must be convex (i.e., satisfy $\sum_{\ell\in{\cal N}_n}c_{n\ell}=1$) in order for the estimates to be unbiased  \cite{Cassio08d,Cattivelli08d}.  Several simple topology-dependent designs have been proposed, such as the uniform rule $c_{n\ell}=\frac{1}{|{\cal{N}}_n|}$,
%
% \begin{equation}\label{eq:Uniform}
%     c_{n\ell}=\frac{1}{|{\cal{N}}_n|}~~\text{,}
% \end{equation}
%
where $|\calN_n|$ is the degree (number of connections) of node $n$; or the Metropolis rule, which employs, for nodes $n$ and $\ell$
\begin{eqnarray}\label{eq:Metropolis}
c_{n\ell}=\begin{cases}
  1/\max(|{\cal{N}}_n|,|{\cal{N}}_\ell|), ~&{\rm if}~ n\neq\ell~ {\rm{are}~ linked;}\\
   0, ~&{\rm for}~ n~ {\rm{and}}~ \ell ~not~ {\rm linked;}\\
  1-\sum_{\ell\in\bar{\cal{N}}_n}c_{n\ell},~~~~&{\rm for}~ n=\ell.
\end{cases}%\right.
\end{eqnarray}
Another rule is based on relative variance \cite{Tu11o}, which assigns $c_{n\ell}$ to be inversely proportional to the noise variance $\sigma^2_{v,n}$:
\begin{eqnarray}\label{eq:error_cov}
c_{n\ell}=\begin{cases}%\left\{\begin{array}{l}
  \frac{\sigma^{-2}_{v,n}}{\sum_{\ell\in {\cal{N}}_n}\sigma^{-2}_{v,\ell}}, &{\rm if}~~ n\neq\ell~~ {\rm{are}~~ linked;}\\
   0, &{\rm for}~ n~ {\rm{and}}~ \ell ~~not~~ {\rm linked.}
\end{cases}%\end{array}\right.
\end{eqnarray}
Such combiners are typically organized into an $N\times N$ matrix $C=[c_{n\ell}]$, which will be  stochastic for the uniform combiner and for \eqref{eq:error_cov}, i.e., for  $\colone = \text{col}[1~~1~~1\dots 1~~1]$ (with length given by the context), we have that $C\colone=\colone$. Combiner (\ref{eq:Metropolis}) leads to a doubly stochastic matrix: $\colone^T C=\colone^T$ and $C\colone=\colone$ \cite{Gallager1995}.

The limited performance of fixed combiners led to the introduction of adaptive combiners $c_{n\ell}(i)$ that are able to account for network diversity and time-varying statistics \cite{Cassio08d,Takahashi10d,Bes12n,Cassio2014,Bes2017}. \textit{Adaptive Diffusion} \cite{Cassio08d} was inspired by parallel-independent combinations of AFs  \cite{Jeronimo06m} and weighs local and neighborhood estimates using fixed combiners $\overline{c}_{n\ell}$, that are convex over $\overline{\calN}_n$, and \textit{one} adaptive  combiner $\lambda_n(i)$ per node. Explicitly,
\begin{equation}
	\label{E:AdaptiveDiffusion}
\begin{aligned}
	\phi_{n,i-1} &= \sum_{\ell \in \overline{\calN}_n} \overline{c}_{n\ell} \, \psi_{\ell,i-1},
	\\
	w_{n,i-1} &= \lambda_n(i) \, \psi_{n,i-1} + [1 - \lambda_n(i)] \phi_{n,i-1},\\
	\psi_{n,i} &= w_{n,i-1} + H_{n,i} u^T_{n,i}(d_n(i)-u_{n,i}w_{n,i-1}),
\end{aligned}
\end{equation}
where the estimate $\phi_{n,i-1}$ fuses estimates $\{\psi_{\ell,i-1}, \ell \in \overline{\calN}_n\}$ from the strict neighborhood via fixed combiners $\{\overline{c}_{n\ell}\}$. The local \add{cooperative filter output} $w_{n,i-1}$ fuses adaptively the local estimate $\psi_{n,i-1}$ with $\phi_{n,i-1}$ via \add{the adaptive parameter} $\lambda_n(i)$, which minimizes the local output error $e_n=d_n(i)-u_{n,i}w_{n,i-1}$ in the mean-square sense.

Later, an alternative approach was taken by adopting $|\calN_n|$ different adaptive combiners per node, one per estimate $\psi_{\ell}$ received from the neighborhood $\calN_n$ \cite{Takahashi10d}. It calculates an approximate $|\calN_n|\times|\calN_n|$ matrix $Q_{n,i-1}^{\text{MSD}}$ to the local (unknown) covariance matrix for the estimate error vector $\psi_{n,i-1}-w^o$, in order to minimize the local $\MSD_n(i)$ w.r.t. the $|\calN_n|$ adaptive combiners captured into the local vector $\{c_{n,i}\}$; this algorithm is referred to as the MSD-algorithm (MSD-alg). Similarly to the Adaptive Diffusion above, a locally fused estimate $\phi_{n,i-1}$ is injected into the local AF. Upon receiving the estimates $\{\psi_{\ell,i-1}\},~\ell\in \calN_n$ from the neighborhood, and letting $S_n^{\text{MSD}} = (I-\frac{\colone \colone^T}{|\calN_n|})$, where $\colone$ is $|\calN_n|\times 1$, the algorithm at node $n$ becomes
\begin{equation}
	\label{E:MSDSupervisor}
\begin{aligned}
    \left[Q_{n,i-1}^{\text{MSD}}\right]_{k\ell}&=[\psi_{\ell,i-1}-\psi_{\ell,i-2}]^T[\psi_{k,i-1}-\psi_{k,i-2}], k, \ell\in \calN_n,\\
    \phi_{n,i-1} &= \sum_{\ell\in \calN_n}[c_{n,i-1}]_\ell\psi_{\ell,i-1},\\
    \psi_{n,i} &= \phi_{n,i-1}+H_{n,i} u_{n,i}^T (d_n(i)-u_{n,i} \phi_{n,i-1}),\\
	c_{n,i} &= c_{n,i-1} - \mu_{M,n} \, S_n^{\text{MSD}} Q_{n,i-1}^{\text{MSD}} c_{n,i-1}\text{,}
\end{aligned}
\end{equation}
% %
where matrix $Q_{n,i-1}^{\text{MSD}}$ is $|\calN_n|\times|\calN_n|$, $[c_{n,i}]_\ell$ is the $\ell$-th element of the combiner vector $c_{n,i}$ and $\mu_{M,n}$ is a scalar step-size that depends on two other parameters $\kappa_M$ and $\eps_M$, according to Equations (14) and (18) in \cite{Takahashi10d}. The initial conditions for the algorithm are $\colone^Tc_{n,-1}=1$, $[c_{n,i}]_\ell \ge 0$, and $\psi_{n,-1}=\psi_{n,-2}=0_{M\times 1}$. Here too both $\psi_{n,i-1}$ or $\phi_{n,i-1}$ can be used to define the output errors\footnote{For the MSD supervisor, we use the latter in Sec. \ref{S:Sims}, since this usually results in better performance.}.

In \cite{Bes12n}, a least-squares (LS) adaptive combiner algorithm (LS-alg) was proposed with an important change in the protocol: keep the local AF estimate $\psi_{n,i-1}$ adapting independently from the rest of the network. In other words, a set of $|\overline{\calN}_n|$ adaptive combiners collected into a local vector $c_{n,i}$ fuses the local independent AF estimate $\psi_{n,i-1}$ with the strict neighborhood estimates $\{w_{\ell,i-1},~\ell\in\overline{\cal N}_n\}$, generating a local estimate $w_{n,i}$ that is ready for use. The procedure is the same as in the original Adaptive Diffusion, with a subtle, yet effective, difference: $w_{n,i}$ is \textit{not} injected into the local AF. The original work had instabilities in the calculation of the combiners, as acknowledged by the authors, so that a stable and better version was published later in \cite{Bes2017}, and is described below. The identity matrix $I$ is $|\overline{\calN}_n|\times|\overline{\calN}_n|$ and $\colone$ is $|\overline{\calN}_n|\times 1$:
\begin{equation}
	\label{E:LSAdaptiveSupervisor}
\begin{aligned}
	\Tilde{y}_{n,i} &= [\, u_{n,i} ( w_{\ell,i-1} - \psi_{n,i-1} ) \,]_{\ell \in \overline{\calN}_n}
		\quad ( |\overline{\calN}_n|\times 1 )
	\\
	e_n(i)&=d_n(i)-u_{n,i}\psi_{n,i-1}
	\\
	P_{n,i} &= \sum_{p = 1}^{i} \gamma_n^{i-p} \, \Tilde{y}_{n,p} \Tilde{y}_{n,p}^T~,\quad z_{n,i} = \sum_{p = 1}^{i} \gamma_n^{i-p} \, \Tilde{y}_{n,p}
	\\
	c_{n,i} &= (P_{n,i}+\epsilon_{LS} {I})^{-1} \, z_{n,i}
	\\
	w_{n,i} &= [1 - \colone^T c_{n,i}] \, \psi_{n,i-1} + \sum_{\ell \in \overline{\calN}_n} [c_{n,i}]_\ell \, w_{\ell,i-1}
		\text{,}\\
	\psi_{n,i} &= \psi_{n,i-1}+ H_{n,i} u^{T}_{n,i} [ d_n(i) - u_{i} \psi_{n,i-1} ],
\end{aligned}
\end{equation}
where $[c_{n,i}]_\ell$ is again the $\ell$-th element of vector $c_{n,i}$, $0\ll\gamma_n<1$ is a local forgetting factor and $\epsilon_{LS}>0$ ensures invertibility in $P_{n,i}+\epsilon_{LS} I$. The estimate $w_{n,i}$ tends to be better than $\psi_{n,i-1}$, thus it is adopted locally. As such, the error quantities \eqref{eq:Local_Errors} for \eqref{E:LSAdaptiveSupervisor} must be updated in terms of $w_{n,i}$.

\add{
The more recent works in \cite{Jin2020} and \cite{Jin2020b} explore directly the original concept of universality at every node. At each iteration, each node runs two independent diffusion process represented by two local adaptive filters whose estimates, $w_{1,n,i}$ and $w_{2,n,i}$, are generated cooperating with nearby nodes and feed a local combiner
\begin{equation}
    w_{n,i} = \lambda_1(i) w_{1,n,i}+\lambda_2(i) w_{2,n,i},
\end{equation}
whose combiners $\lambda_1(i)$ and $\lambda_2(i)$ are chosen to be either convex \cite{Jin2020b}, or affine \cite{Jin2020}. The output estimate $w_{n,i}$ is guaranteed to be classically universal in the mean-square error sense. As such, each node output is at least as good as its two input diffusion processes $w_{1,n,i}$ and $w_{2,n,i}$. The overall result is that the output network average performance is at least as good as the best average network performance between the two input diffusion processes. This important contribution is not universal in any of the proposed universality types introduced in this work. This is because the output estimates $w_{n,i}$ at every node are locally confined, they are not propagated to the neighborhood: there is no network level feedback and no network learning takes place, in the sense introduced in \cite{Cassio2014}.
}

Extensive simulations show that both algorithms (\ref{E:MSDSupervisor}) and (\ref{E:LSAdaptiveSupervisor}) consistently outperform both the original diffusion LMS \cite{Cassio08d} and its adaptive version (\ref{E:AdaptiveDiffusion}). Although effective at improving performance, (\ref{E:MSDSupervisor}) and (\ref{E:LSAdaptiveSupervisor}) are \textit{not} distributed universal algorithms as the one proposed in \cite{Cassio2014}, which is extended and studied in detail in this work.

\vspace{-0.2cm}
\section{Distributed Universal Estimation}\label{sec:D_U_AN}

Our first task is to extend the concepts of universal estimation  \cite{Merhav98u,Singer99u} to the distributed case.

The design of an AF takes place by optimizing some figure of merit, say mean-square error, in terms of a set of parameters $\theta$, which may include the filter order $M$, the step-size $\mu$, the forgetting factor $\gamma$ for the RLS filter, the rank parameter $K$ for the APA filter, etc. Traditionally, such parameters must be designed from not always accurate analytical models that may depend on unknown quantities, so that empirical tests must be made. When the scenario is time-varying, the design is an even more challenging task.

In this context,  universal estimation is a change in the design paradigm: instead of choosing a fixed parameter set $\theta$ under limited knowledge, select a pool of $K$ candidates $\{\theta_k\}$ for $\theta$, each of them forming an individual estimator, or expert, and present them to a supervisor, which is ultimately in charge of generating a reasonable estimate $w$ for $w^o$ by consulting the pool of experts. The central question is: what is a reasonable estimate? The answer to that is the very definition of universal estimation: an estimator is considered universal when its supervisor is able to at least match the performance of the best individual expert in the pool, in terms of the adopted figure of merit.  Within a class, the pool $\{\theta_k\}$ has to be rich enough to span the unknown optimal $\theta$. Of course, this might represent an explosion in computational complexity, so care must be taken to find a trade-off.
The supervisor is any function that consults the pool of experts and delivers a universal estimate w.r.t. the adopted figure of merit.

Such ideas have been extensively and successfully explored in the literature in the context of combinations of adaptive filters \cite{Jeronimo06m,Chambers06c}, in terms of pools of filters with different step-sizes, filter orders, or even different learning rules \cite{Azpicueta08n,Magno08i,Vitor09t,Cassio10a,Chamon12c}.
The supervisor admits different designs, but one that is widely adopted and is efficient is the convex supervisor, which is an activation function of a free parameter to be adapted \cite{Cassio10a}.

\subsection{Distributed universality}\label{ssec:Defs_Universality}

\add{Typical distributed adaptive systems rely on several nodes consulting multiple data sources across the field of interest, therefore subject to a natural spatial diversity. In general, cooperation among the nodes is desirable, though it has to be implemented under limited communications. The combination of spatial diversity and limited communications may drive different nodes to different performances, which might not be acceptable in most applications. As such, a intuitive set of desired properties for ANs may be defined and will guide the development of a distributed universal protocol to operate in any network:}
\begin{enumerate}
    \item[A)] Ability to reject a bad node;
    \item[B)] Ability to exploit an exceptional node;
    \item[C)] Node performance homogeneity.
\end{enumerate}

Rejecting a bad node (A) relates to the ability of avoiding using poor estimates that could degrade the performance of other nodes in the network. This is paramount in applications such as remote sensing networks, where sensor damage can degrade the information provided by nodes. Similarly, the ability to exploit estimates from an exceptional node (B) lends robustness to the AN: nodes operating in poor conditions can take advantage of better nodes. Finally, performance homogeneity (C) is fundamental if the network is to operate in a fully distributed manner.  Indeed, each node must ultimately rely on its local estimate to perform actions, such as alerting to the presence of an intruder or anomalies in the field \cite{Cattivelli11d,Allan2018}, or even controlling some environmental variable \cite{Estrin2002,Plata2015}. In many ways, promoting good performance across all nodes is more important than having better average performance with some nodes (much) worse than others.

In order to address the aforementioned desired properties, three main points must be tackled: define the experts; select figures of merit; and \add{define what is a supervisor when multiple data sources are consulted and cooperation among nodes is a requirement.}

In the AN case, the selection of the pool of experts is natural and is comprised of the $N$ adaptive nodes attempting to estimate $w^o$ from the data available across the network. The required expert diversity is guaranteed by the natural space-time data diversity, and/or by employing different AF parameters at the nodes, or even different learning rules.

Typical figures of merit  adopted for ANs are global network quantities: average MSE, EMSE and MSD, defined from the local quantities previously presented in \eqref{eq:Local_Errors}:
\begin{equation}\label{eq:Network_Errors}
    \begin{aligned}
    \MSE(i) &= \frac{1}{N}\sum_{n=1}^N \MSE_n(i),\\
    \EMSE(i) &= \frac{1}{N}\sum_{n=1}^N \EMSE_n(i)~~\text{(Network quantities),}\\
    \MSD(i) &= \frac{1}{N}\sum_{n=1}^N \MSD_n(i).
    \end{aligned}
\end{equation}
These metrics provide a valid overall measure of network performance and are widely used in the literature, with or
without cooperation. Cooperation does improve network performance according to these metrics, even if not everybody in the network will experience an improvement, as compared to the non-cooperative operation \cite{Cassio2008b,Bes12n}. This implies that we must look carefully into the way nodes cooperate, and should discriminate performance across the network by also inspecting the local quantities $\MSE_n(i)$, $\EMSE_n(i)$ and $\MSD_n(i)$.

One key point of universal estimators is that expert integrity is preserved, which implies that at some level they should work independently. This is because cooperation might lead to the poor performance of some experts to contaminate that of the good ones, and who's whom is application-dependent or may vary over time. In the AN case, this property is usually violated when cooperation is implemented. On the other hand, without cooperation the global performance of ANs is deterred to a great extent; furthermore, in some applications, such as source localization and scalar field estimation applications, nodes \textit{must} cooperate in order to solve the global problem in a distributed manner \cite{Yannick2012,Scabin2016}. Therefore, in order to enforce the desired properties of ANs, cooperation demands that the definition (def.) of universality be extended.

\begin{defn}[Local universality]
	\label{D:LocalUniversal}

	A node $n$ of an AN is said to be locally universal when it holds that $n$ is at least as good as the best node in its neighborhood, i.e{.}, the best node $m \in \calN_n$.

\end{defn}
\begin{defn}[Global universality]
	\label{D:GlobalUniversal}

	An AN is said to be globally universal when for all nodes $n \in \{1,\dots,N\}$ it holds that $n$ is as good as the best node in the network, i.e., the best node $m \in \{1,\dots,N\}$.

\end{defn}
\begin{defn}[Universality w.r.t. the non-cooperative strategy]
	\label{D:NonCoopUniversal}

	An AN is said to be universal with respect to the non-cooperative strategy when for all nodes $n \in \{1,\dots,N\}$ it holds that node $n$ performs at least as well as if it were independent of the rest of the network.

\end{defn}
\begin{defn}[Asymptotic universality]
	\label{D:AsymptoticUniversal}

	An AN is said to be asymptotically universal, when it is universal for $i \to \infty$, i.e{.}, at steady-state.

\end{defn}

These definitions are inspired by those found in the contexts of universal prediction~\cite{Merhav98u} and theory of individual sequences~\cite{CesaBianchi99p}. They are straightforward and intuitive, although nontrivial. For instance, Def.~\ref{D:NonCoopUniversal} might seem obvious, but it is often violated in standard diffusion ANs~\cite{Cassio08d}, as is the case of Def.~\ref{D:LocalUniversal}.

\add{Figure \ref{fig:UniversalDefs} depicts examples in a hypothetical four-node connected network to illustrate the different definitions of universality.} Considering asymptotic behavior~(Def.~\ref{D:AsymptoticUniversal}), it is straightforward to see that local universality~(Def.~\ref{D:LocalUniversal}) implies both rejection of bad nodes and the exploitation of exceptional ones. Moreover, for a connected AN with an undirected graph, if every node is locally universal, then global universality~(Def.~\ref{D:GlobalUniversal}) follows (if the network is undirected and locally universal, then the performance of each pair of connected nodes must be equal.  If the network is connected, there is a path between every pair of nodes; thus the network must be globally universal). Thus, globally universal networks not only guarantee that underperforming nodes are isolated, but also that all nodes take advantage of the performance of superior ones. What is more, global universality guarantees node performance homogeneity. Finally, \add{rejecting poor nodes} is clearly related to the concept of universality w.r.t{.} the non-cooperative strategy~(Def.~\ref{D:NonCoopUniversal}), as it requires that cooperation only improves local performance.

Notice that Definitions~\ref{D:GlobalUniversal} and~\ref{D:NonCoopUniversal} denote two different forms of distributed universality. Indeed, it is possible for ANs to be globally universal without being universal w.r.t{.} the non-cooperative strategy and vice-versa. In fact, global universality alone only leads to performance homogeneity across nodes. Universality w.r.t{.} the non-cooperative strategy, on the other hand, guarantees that cooperating is the best strategy for each individual node in the network instead of only for the network on average. Both concepts must apply to obtain all the aforementioned properties for ANs. Definition \ref{D:AsymptoticUniversal} is a realistic definition, in view of the limited communications and energy within the network: in a large network, with a large radius \cite{Bollobas1998}, an exceptional node estimate that is several hops away from another given node will take several algorithm cycles to be broadcast network-wide, making universality necessarily an asymptotic property.

\begin{figure}[tb!]
	\begin{minipage}[c]{0.32\columnwidth}
		\centering
		\includegraphics[width=\columnwidth]{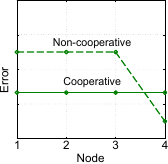}
		\small{(a)}
	\end{minipage}
	\hfill
	\begin{minipage}[c]{0.32\columnwidth}
		\centering
		\includegraphics[width=\columnwidth]{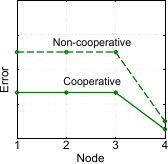}
		\small{(b)}
	\end{minipage}
	\hfill
	\begin{minipage}[c]{0.32\columnwidth}
		\centering
		\includegraphics[width=\columnwidth]{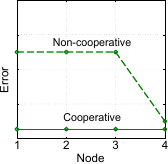}
		\small{(c)}
	\end{minipage}
	\vspace*{-3pt}
	\caption{Distributed asymptotic universality (Def. \ref{D:AsymptoticUniversal}) in ANs:
				(a)~Globally universal~(Def.~\ref{D:GlobalUniversal}) but not universal w.r.t{.} the non-cooperative strategy~(Def.~\ref{D:NonCoopUniversal});
				(b)~Universal w.r.t{.} the non-cooperative strategy~(Def.~\ref{D:NonCoopUniversal}), but not globally universal~(Def.~\ref{D:GlobalUniversal});
				(c)~Globally universal~(Def.~\ref{D:GlobalUniversal}) and universal w.r.t{.} the non-cooperative strategy~(Def.~\ref{D:NonCoopUniversal}).
				}
		\label{fig:UniversalDefs}
\end{figure}

\subsection{\add{The quest for a distributed supervisor}}\label{ssec:Dist_Supervisors}

\add{The first point to make is that a single supervisor directly imported from the standard single data source case \cite{Jeronimo06m,Chambers06c,Cassio10a} is not viable to implement distributed universality. This is because in a network with $N$ distributed experts consulting $N$ data sources, the access of the single supervisor to all the expert estimates would require its centralization: nodes report their estimates to the sole supervisor which promptly reports back a universal estimate for local use\footnote{\add{For instance, the MSE for each expert could be calculated and the supervisor would select the estimate corresponding to the expert with the smallest MSE.}}. Here universality is always guaranteed and its original definition is enough, since the network will access the best estimate and it will be the same for everybody. However, a central node is highly undesirable as it represents a catastrophic failure point for the network and the amount of communication resources (energy, more powerful transceivers, etc.) is prohibitively large; and it might be infeasible simply due to lack of connectivity, either because links are not available, or because routing protocols may impose too much overhead and excessive delays.}

\add{A second strategy would be to implement a fully connected network. In this setup, every node may be seen as a central node, and this hints at placing $N$ supervisors, one per node, embedded with the local expert. A given supervisor at node $k$ receives estimates from all the network experts and correctly generates the best estimate for local use. Due to the full connectivity, this may be replicated at all other nodes with their respective local supervisors, so that it is guaranteed that every node will have an estimate that is the best the network has to offer and all estimates will be statistically equivalent, since every node has access to the same set of experts. As a consequence, the original universality concept suffices too in this scheme. However, the implementation of this level of connectivity demands even larger resources than the centralized node case. Thus, it is also impractical.}

\add{We are left with the realistic scenario, mentioned earlier in this section: nodes have access only to their nearby peers, with different degrees of connectivity and with multiple data sources that reflect spatial diversity: these facts combined do not guarantee in general that a node will never loose from cooperation, that the network has the ability to reject bad nodes and/or promote exceptional ones, or even that the node performance will be uniform. In this scenario, we establish the main conceptual leap from standard universal estimation: define $N$ supervisors, one per node, and let the \emph{supervisors cooperate}, instead of direct expert cooperation. In summary, we propose that a \textit{distributed supervisor} may be implemented by a set of $N$ collaborative supervisors that shelter their local experts from the outer world. In such realistic scenarios, existing distributed adaptive systems may present, or not, the different kinds of universality formalized earlier. In response to that, in the following section we show in detail how to construct a cooperation protocol that explores the idea of distributed supervision, and that is proven to be universal w.r.t{.} all the Definitions \ref{D:LocalUniversal}--\ref{D:NonCoopUniversal}.
}

\section{A distributed universal cooperation protocol}\label{sec:U_protocol}

Why should a node cooperate, if its performance might deteriorate? Should cooperation in the name of a greater good be enough? Such questions lie at the heart of \add{what distributed supervision should deliver in the distributed multisource case.}

To begin with, a more inviting cooperation protocol should guarantee that a node never loses from cooperating; its performance at least remains the same, as compared to its independent operation. The idea is that all nodes improve, with the possible exception of the best node
in the network, which should not worsen its non-cooperative performance. Due to topology constraints, a generic node is unaware of which, or where, is the best node in the network. If the protocol assures that there will be no losses in performance, then cooperating becomes an interesting deal.

\subsection{Constructing the protocol}

A protocol that promotes distributed universality with respect to the introduced Definitions \ref{D:LocalUniversal}--\ref{D:NonCoopUniversal} has two steps.

Firstly, notice that the concept of universality w.r.t{.} the non-cooperative strategy~(Def.~\ref{D:NonCoopUniversal}) motivates the idea of protecting local estimates from network perturbations, an idea that was also proposed in~\cite{Bes12n}, albeit for different reasons. Indeed, allowing each node to operate as if it were independent of the network is a simple way to guarantee that its estimation process is not disturbed by underperforming neighbors. This, however, leads to a non-cooperative network. Hence, the nodes need a way to preserve their own estimates, without neglecting those from the rest of the network. This can be implemented much like what was done in the adaptive diffusion scheme (\ref{E:AdaptiveDiffusion}) \cite{Cassio08d}, in terms of a local independent estimate $\psi_n$ and
an estimate $\phi_n$ fused from neighboring \add{supervisor} estimates
\begin{subequations}\label{eq:d-universal-AN}
\begin{eqnarray}
    \hspace{-0.7cm}\phi_{n,i-1} &=& \sum_{\ell \in \overline{\calN}_n} \overline{c}_{n\ell} w_{\ell,i-1},	\label{eq:d-universal-AN1}\\
	w_{n,i} &=& \lambda_n(i) \, \psi_{n,i-1} + [1 - \lambda_n(i)] \phi_{n,i-1},\label{eq:d-universal-AN2}\\
	\psi_{n,i} &=& \psi_{n,i-1} + H_{n,i} u^T_{n,i}(d_n(i)-u_{n,i}\psi_{n,i-1}).
\end{eqnarray}
\end{subequations}
Two crucial differences in (\ref{eq:d-universal-AN}) from the adaptive diffusion (\ref{E:AdaptiveDiffusion}) stand out: (a) the local estimate $\psi_{n,i}$ evolves independently according to the local learning rule, i.e., $w_{n,i}$ \textit{is not injected} into the local AF \cite{Bes12n}, \cite{Cassio2014}; and (b) the local supervisor estimate $w_{n,i-1}$ is shared within its neighborhood, instead of sharing $\psi_{n,i-1}$. Intuitively, $w_{n,i-1}$ tends to be a better estimate than $\psi_{n,i-1}$. Furthermore, in \cite{Cassio2014}  a network learning model was developed that shows how sharing $w_{n,i-1}$ implements a network-level feedback that, hop by hop, allows the best estimates to reach the whole network, while sharing $\psi_{n,i-1}$ not necessarily does. Finally, (\ref{eq:d-universal-AN}) unveils the need for one supervisor per node that implements cooperation, rather than the local adaptive filters: cooperation is carried out indirectly.

The last resource is an accelerating mechanism
that feeds back the supervisor output into the local AF, every $L_n$ iterations \cite{Chamon12c}; it was
already successfully applied in \cite{Cassio2014}:
\begin{eqnarray}\label{eq:Cyclic-Feedback}
\psi_{n,a} &=& \delta_{L_n}(i) \, w_{n,i} + \left[ 1 - \delta_{L_n}(i) \right] \psi_{n,i-1},
	\\
\psi_{n,i} &=& \psi_{n,a} + H_{n,i} u^{T}_{n,i} [ d_n(i) - u_{n,i} \psi_{n,a}]\text{,}
\end{eqnarray}
where $\delta_{L_n}(i)=\delta(i - r L_n)$ is the Kronecker delta, with $r\in\mathbb{Z}^+$. Note that an adaptive diffusion iteration, as in (\ref{E:AdaptiveDiffusion}), is periodically implemented above: within an $L_n$-length cycle, the local AF experiences $L_n-1$ independent iterations; at the $L_n$-th iteration, it is perturbed with the supervisor estimate $w_{n,i}$ exactly once, so that the local AF has access to the potentially best estimate in its neighborhood. As such, $L_n$ should not be small, since otherwise it violates the principle of preserving the local AF. On the other hand, too large values for $L_n$ do not accelerate the transient, returning to purely isolated local estimates in the limit $L_n\to\infty$ \add{(i.e., transfers never occur)}. A simple design technique for $L_n$ is imported from \cite{Chamon12c}, and a typical value for AN applications is $L_n=1,000$. \add{We note, however, that  a good value for $L_n$ may be larger if the input signals are very correlated or if $M$ is large, such that convergence of the local experts is slow.  Conversely, smaller values of $L_n$ might be useful for small $M$ and uncorrelated signals, a situation in which the local experts will converge quickly. In Section \ref{S:Sims} we show in Example 6 (See Fig. \ref{fig:msdmax}) that the network performance is relatively insensitive within a wide range for
$L_n$.}

We now collect all the equations into a complete \textit{distributed universal adaptive network} with a generic learning rule at the nodes. Upon selecting the local learning rule via matrix $H_{n,i}$, and the reception of the supervisor estimates $\{w_{\ell,i-1}\}$ from the neighborhood, the proposed algorithm \add{implemented at node $n$} is:
%\begin{subequations}\label{E:UniversalSupervisor}
\begin{subequations}\label{E:UniversalSupervisor}
\allowdisplaybreaks
\begin{align}
	\lambda_n(i) &= \frac{1}{1 + e^{-a_n(i-1)}}, \label{E:US_eta}\\
	\check{\lambda}_n(i) &= \begin{cases}
	\lambda_n(i),  &\text{ if } -a_+ < a_n(i)< a_+,\\
	0,  &\text{ if }a_n(i) = -a_+,\\
	1, & \text{ if }a_n(i) = a_+.
	\end{cases}\label{E:US_etacheck}
	\\
	\phi_{n,i-1} &= \sum_{\ell \in \overline{\calN}_n} \overline{c}_{n\ell} \, w_{\ell,i-1},\label{E:US_combine}
	\\
	w_{n,i} &= \check{\lambda}_n(i) \, \psi_{n,i-1} + [1 - \check{\lambda}_n(i)] \phi_{n,i-1},\label{E:US_outputFilter}
	\\
	e_n(i) &= d_n(i) - u_{n,i} w_{n,i},	\label{E:US_local_error}
	\\
	p_n(i) &= \nu_n p_n(i-1) \notag\\
  &+ [1-\nu_n] \abs{u_{n,i} (\psi_{n,i-1} - \phi_{n,i-1})}^2,
	\\
	\tilmu_{a,n} &= \mu_{a,n} \,/\, [p_n(i) + \eps_p],\label{E:US_tilmu}
	\\
	a_n(i) &= \bigl[a_n  \notag\\
 &+ \tilmu_{a,n} u_{n,i} ( \psi_{n} - \phi_{n} ) e_n(i)\lambda_n [1 - \lambda_n]\bigr]_{-a_+}^{a_+},\label{E:US_a}
% 	a_n(i) &= \left[a_n(i-1) + \tilmu_{a,n} u_{n,i} ( \psi_{n,i-1} - \phi_{n,i-1} ) e_n(i)\times\lambda_n(i) [1 - \lambda_n(i)]\right]_{-a_+}^{a_+},\label{E:US_a}
	\\
	\psi_{n,a} &= \delta_{L_n}(i) \, w_{n,i} + \left[ 1 - \delta_{L_n}(i) \right] \psi_{n,i-1},	\label{E:US_FB}
	\\
	\psi_{n,i} &= \psi_{n,a} +  H_{n,i}u^{*}_{n,i} [ d_n(i) - u_{n,i} \psi_{n,a} ]\text{,}\label{E:US_adapt}
\end{align}
\end{subequations}
where in \eqref{E:US_a} $a_n=a_n(i-1)$, $\lambda_n=\lambda_n(i)$, $\psi_{n}=\psi_{n,i-1}$ and $\phi_{n}=\phi_{n,i-1}$.

\add{In the algorithm above, \eqref{E:US_eta} is a convex activation function that represents the supervisor parameter, which is adapted in terms of the auxiliary variable $a_n(i)$ \cite{arenas-garciaCombinationsAdaptiveFilters2016}. Eq. \eqref{E:US_etacheck} implements a truncation operation, either ceiling $\lambda_n(i)$ to $1$, or flooring it to $0$, depending on the limiting parameter $a_+$ (Typically $a_+=4$); this results in a smaller variance for the random variable $\check{\lambda}_n(i)$,  also accelerating convergence \cite{arenas-garciaCombinationsAdaptiveFilters2016}.
The neighborhood supervisor estimates are fused into $\phi_{n,i-1}$ in \eqref{E:US_combine}, which is used to generate the local supervisor output $w_{n,i}$ in \eqref{E:US_outputFilter} for local use, with the associated estimation error $e_n(i)$ in \eqref{E:US_local_error}. The quantity $p_n(i)$ is a normalization factor, with the associated filtering parameter $0\ll\nu_n<1$ and regularization parameter $0<\epsilon_p\ll 1$, that helps improving the convergence of the parameter $a_n(i)$; this also has the effect of considerably limiting the required range for $\mu_{a,n}$ in \eqref {E:US_tilmu}, which typically can be chosen in the interval $(0, 1]$ when normalization is adopted \cite{arenas-garciaCombinationsAdaptiveFilters2016}. Equation \eqref{E:US_a} is the actual update recursion for $a_n(i)$ and it drives $\lambda_n(i)$ in \eqref{E:US_eta}.}

Since (\ref{E:UniversalSupervisor}) above evolved from the adaptive diffusion protocol, the fixed combiners $\{\overline{c}_{n\ell}\geq0\}$ must also be convex over the strict neighborhood $\overline{\calN}_n$, that is $\sum_{\ell} \overline{c}_{n\ell}=1$ for $0 \le n, \ell \le N$, and $c_{n\ell}=0$ for $\ell\neq \overline{\calN}_n$. For that matter, the Uniform, (\ref{eq:Metropolis}) and (\ref{eq:error_cov}) rules may be adopted to design $\{\overline{c}_{n\ell}\}$, mutatis mutandis. The Universal Adaptive Supervisor (\ref{E:UniversalSupervisor}) is referred to as the \textit{U-sup} algorithm. As with the LS-alg (\ref{E:LSAdaptiveSupervisor}), the best available estimate at node n is $w_{n,i}$, thus the error definitions for U-sup must be updated w.r.t $w_{n,i}$.
%%%
Besides performance, the computational complexity is of central importance in IoT and sensor network applications, and here we consider the number of multiplications per node per iteration $N_\times$ as the metric for comparison, disregarding the local AF operations (which are essentially the same for all distributed algorithms considered here).\footnote{Nevertheless, they still play a role in the computations for cooperation strategies.} The number of multiplications required for implementing the Universal supervisor algorithm (U-sup) (\ref{E:UniversalSupervisor}), the Mean-Square Deviation combiners algorithm (MSD-alg) (\ref{E:MSDSupervisor}) and the Least-Squares combiners algorithm (LS-alg) (\ref{E:LSAdaptiveSupervisor}), respectively, are
\begin{equation}\label{eq:complexity}
    \begin{aligned}
    N_\times \text{(U-sup)} &= \big(|\calN_n|+3\big)M+7\\
    N_\times \text{(MSD-alg)}&=\frac{\big(|\calN_n|^2+3|\calN_n|+2\big)}{2}M+2|\calN_n|^2\\
    N_\times \text{(LS-alg)} &= 2\big(|\calN_n| + 1\big)M + |\calN_n|^3/3+|\calN_n|^2+|\calN_n|~\text{.}
    \end{aligned}
\end{equation}
The complexity of all algorithms depends, obviously, on the application, which is captured by $M$; but also depends on the network topology, captured by $|\calN_n|$ (which assumes the algorithms explore the entire neighborhood at each node).
For a given application, which means a \add{fixed $M$}, if the network infrastructure is enlarged for performance improvement, the U-sup complexity scales linearly with $|\calN_n|$; the LS-alg \add{has a linear term on $|\calN_n|M$, but which is twice as complex as the corresponding term in U-sup,} and has a cubic term $|\calN_n|^3/3$ that is application independent; and the MSD-alg scales quadratically with $|\calN_n|$, and has another quadratic term $2|\calN_n|^2$ that is also application independent. As a numerical example, consider \textbf{Example 1} in Section VI: a network of $N=15$ nodes, with $M=50$ for the AFs, and average node degree of  $|\calN_n|=6$, returns  457 multiplications for the U-sup algorithm, 814 for the LS-alg, and 1472 for the MSD-alg. For the same example, increasing node degree to $|\calN_n|=10$, the U-sup will require 657 multiplications (44\% increase), 1543 for the LS-alg (90\% increase) and 3500 for the MSD-alg (138\% increase).

\subsection{Universality of the proposed protocol}\label{S:UniversalityProof}

Showing that algorithm (\ref{E:UniversalSupervisor}) achieves the Definitions \ref{D:LocalUniversal}--\ref{D:NonCoopUniversal} of distributed universality is intricate, since \eqref{E:UniversalSupervisor} is a set of stochastic coupled nonlinear recursions; in particular, the recursions for the local supervisors $\lambda_n(i)$ are coupled due to the sharing of information across the network. In this section, we show that any steady-state solution to \eqref{E:UniversalSupervisor} must achieve universal performance according to Definitions \ref{D:LocalUniversal}--\ref{D:NonCoopUniversal}, under three assumptions:
\begin{enumerate}[label=A.\arabic{enumi}]
    \item\label{A:ss} The network is at steady-state\add{, that is, all local filters have converged to their final MSD performance, and the supervisors have also converged to their final values (see Sec. \ref{sec:Analysis} for a discussion about convergence)};
    \item\label{A:msd} The local supervisors \eqref{E:US_etacheck}, \eqref{E:US_outputFilter} choose the best option (in terms of MSD) between $\psi_{n,i}$ and $\phi_{n,i}$ at steady-state;
    \item\label{A:Ln} The local filters are independent, i.e., $L_n\rightarrow\infty$.
\end{enumerate}
Note that Assumption \ref{A:msd} was also used in \cite{Jeronimo06m} to prove that the convex combination scheme is universal.  This assumption is justified since using \cite[eq. (11) and (17)]{Vitor09t} it can be shown that the supervisor weight for the convex combination scheme minimizes the combination MSE if $\mu_a\rightarrow 0$.  Minimization of MSE is equivalent to minimization of the MSD when the regressors $u_{n,i}$ are white \cite{Sayed08a}.  Note also that Assumption \ref{A:msd} is related to local properties of the local supervisors, and is thus not equivalent to assuming network universality. \add{Assumption \ref{A:Ln} is equivalent to requiring that $L_n$ is large enough so that the local filters and supervisors have time to converge before a decision about transfer of coefficients is made.}

In the next section we present a model for the transient behavior of \eqref{E:UniversalSupervisor} in the mean and mean-square senses, but the study of the limiting behavior of the resulting model is still considerably difficult, and is left for a future work.

We now explore the steady-state properties of the proposed AN distributed estimator (\ref{E:UniversalSupervisor}). We show in the next two theorems that the proposed scheme is universal w.r.t. the non-cooperative strategy, and that, if a network reaches steady-state (constant MSDs at each node), then necessarily the supervisor leads to global universality.

\begin{Theorem}[Universality w.r.t{.} the non-cooperative strategy]
	\label{T:NonCoopUniversal}
Under Assumptions \ref{A:ss}--\ref{A:Ln}, the network feedback protocol described in~\eqref{E:UniversalSupervisor} is asymptotically universal w.r.t{.} the non-cooperative strategy~(Def.~\ref{D:NonCoopUniversal}).

\end{Theorem}
\begin{proof}
From equation~\eqref{E:US_outputFilter}, the output $w_{n,i}$ of node $n$ is a linear combination between the non-cooperative local estimate~$\psi_{n,i-1}$ and the averaged estimates from its neighborhood~$\phi_{n,i-1}$. Notice that the non-cooperative strategy is a particular case of~\eqref{E:UniversalSupervisor} in which~$\lambda_n(i) = 1$, for all $i$. Thus, since the linear combiner~$\lambda_n(i)$ minimizes the local $\MSD_n(i)$, the output of each node  $w_{n,i}$ is guaranteed to be at least as good as its non-cooperative version. If the local estimate $\psi_{n,i-1}$ is better than $\phi_{n,i-1}$, then the supervisor $\lambda_n(i)$ will drive $w_{n,i}$ to at least the local performance; on the other hand, if $\phi_{n,i-1}$ is better, then $\lambda_n(i)$ will guide the local node output $w_{n,i}$ to the average neighborhood performance, which is better than the local non-cooperative by hypothesis. Therefore, node $n$ never loses from cooperating, thus it is universal w.r.t. the non-cooperative case (Definition \ref{D:LocalUniversal}).
\end{proof}
\vspace{0.2cm}

\begin{Theorem}[Asymptotic global universality] \label{thm:Global_universal} The network feedback protocol from (\ref{E:UniversalSupervisor}) is globally universal (Definitions \ref{D:GlobalUniversal} and \ref{D:AsymptoticUniversal}) under Assumptions \ref{A:ss}--\ref{A:msd}.
\end{Theorem}
\begin{proof}
Suppose that the estimates $w_{n,i}$ computed by each node resulted in different values of local $\MSD_n(i)=\E\|w^o-w_{n,i}\|^2$ at steady state.  We show next that this leads to a contradiction.  Assume then that the network is at steady-state and node $n_0$ has the worst performance of all nodes, that is, $\MSD_{n_0}(i) \ge \MSD_{n}(i)$ for all $n\neq n_0$ and $\MSD_{n_0}(i) > \MSD_{\ell}(i)$ for at least one $\ell \in \overline{\calN}_{n_0}$ (such a node has to exist, since the number of nodes is finite and we are assuming that not all local MSDs are equal). Therefore, we have, at steady-state,
\begin{equation}
    \begin{aligned}
        \MSD_{n_0}(i)&=\E\|w_{n_0,i}-w^o\|^2 \ge \E\|w_{n,i} - w^o\|^2\\
        &=\qquad\MSD_n(i) \text{ for } n \in \overline{\calN}_{n_0},
    \end{aligned}
\end{equation}
with $\MSD_{n_0}(i) > \MSD_{\ell}(i)$ for at least one $\ell\in \overline{\calN}_{n_0}$.
Now, at steady-state $\MSD_n(i) = \MSD_n(i-1)$, and  from \eqref{E:US_combine}
\begin{equation}\begin{aligned}
&\E\|\phi_{n_0,i-1}-w^o\|^2=\E\biggl\|\sum_{\ell\in\overline{\calN}_{n_0}} \overline{c}_{n_0\ell}w_{\ell,i-1}-w^o \biggr\|^2 \\
&\le \E \sum_{\ell\in\overline{\calN}_{n_0}}\overline{c}_{n_0\ell}\|w_{\ell,i-1}-w^o\|^2 = \sum_{\ell\in\overline{\calN}_{n_0}}\overline{c}_{n_0\ell} \MSD_\ell(i-1),
\end{aligned}\label{U:thm.inequality}\end{equation}
where we used the fact that $\|\cdot\|^2$ is a convex function and the $\{\overline{c}_{nn_0}\}$ add up to one.  Since by hypothesis $n_0$ is the worst node, and at least one node in its neighborhood has a smaller MSD, we conclude that necessarily
\begin{equation}
    \E\|\phi_{n_0,i-1}-w^o\|^2 \le \sum_{\ell\in\overline{\calN}_{n_0}}\overline{c}_{n_0\ell} \MSD_\ell(i-1) < \MSD_{n_0}(i-1).\label{U:MSD.inequality}
\end{equation}
The last inequality results from the hypotheses that $\sum_{\ell\in \overline{\calN}_{n_0}}\bar{c}_{n_0\ell} = 1$, that $\MSD_n\le\MSD_{n_0}$, and that there is $\ell_0\in\overline{\calN}_{n_0}$ such that $\MSD_{\ell_0}<\MSD_{n_0}$.
This means that the supervisor for node $n_0$ should change its choice to reduce the MSD, contradicting Assumptions \ref{A:ss}-\ref{A:msd}. We conclude therefore that at steady-state all nodes must have the same performance.
\end{proof}

\section{Performance Analysis}
	\label{sec:Analysis}

In this section we propose a model for the mean and mean-square performance of the new \add{universal} diffusion strategy. Let us start by introducing some additional assumptions. We extend our data model (\ref{eq:data_model}), now allowing for changes in the vector of unknown parameters:
\begin{equation}\label{eq:data-model-tv}
	d_n(i) = u_{n,i} w^o_{i-1} + v_n(i)
		\text{,}
\end{equation}
where  $w^o_{i-1} \in \mathbb{R}^{M \times 1}$ is a \add{time-varying} vector of unknown parameters the network is trying to estimate, $v_n(i)$ is an i.i.d{.} zero-mean measurement noise with variance $\sigma_{v,n}^2$, independent of all regressor vectors $\{u_{\ell,i}\}$ in the network. The initial condition $w^o_{-1}$ is a random unit norm vector. We further assume that
\begin{enumerate}[resume*]
\item\label{A:indep} $\{ u_{n,i} \}$ is a zero-mean i.i.d{.} sequence with covariance matrix $R_n$;
\item\label{A:Ln2} $L_n\equiv L$ is the same for all nodes in the network;
\item\label{A:random.walk} The optimum parameter vector $w^o$ may change according to a random walk model
\begin{equation}\label{eq:RandomWalk}
    w^o_i = w^o_{i-1} + q_i,
\end{equation}

where $\{q_i\}$ is an i.i.d. vector sequence with zero mean and autocovariance matrix $Q = \E q_i q_i^T$;
\item\label{A:smallstepsize} \add{The stepsizes $\mu_n$  and $\mu_{a,n}$ are small enough for the usual slow adaptation approximations in adaptive filtering to be valid, and such that the variance of $a_n(i)$ can be disregarded   \cite{Sayed08a,Diniz13a,Nascimento2014,arenas-garciaCombinationsAdaptiveFilters2016};}
\item\label{A:nu} \add{The forgetting factors $\nu_n$ are close to one, so that the variance of $p_n(i)$ can be disregarded;}
\item\label{A:etacheck} \add{$\check{\lambda}_n(i)=\lambda_n(i)$ always in \eqref{E:US_etacheck}.}
\end{enumerate}
The discussion below assumes, for simplicity, that the local adaptive filters in all nodes are using the same  algorithm, either LMS or NLMS. It would not be difficult to modify the models and arguments for other types of filters, or even for networks running different classes of filters at each node. Assumptions \ref{A:indep},  \ref{A:random.walk}\add{, \ref{A:smallstepsize} and \ref{A:nu} }are widely used in the literature \cite{Sayed08a,Diniz13a,arenas-garciaCombinationsAdaptiveFilters2016}.  Assumption \ref{A:Ln2} is used only to simplify the analysis and can be easily relaxed.  \add{Assumption \ref{A:etacheck} is used to simplify the model and will tend to increase the model variances.}

\subsection{The global adaptive network model}

We proceed by defining the local error vector quantities for each node as
\begin{align}\label{eq:Psitl}
	\Psitl_{n,i} &= w^o_i - \psi_{n,i} \text{,} &
	\tilw_{n,i} &= w^o_i-w_{n,i}
		\text{.}
\end{align}
Next, collect the local quantities defined in Algorithm (\ref{E:UniversalSupervisor}) into global variables:
\begin{align*}
	\Psitl_i &= \col \left( \Psitl_{n,i} \right) \text{,} &
	\tilw_i &= \col \left( \tilw_{n,i} \right) \text{,} &
	v_i &= \col \left( v_n(i) \right) \text{,}
	\\
	e_i &= \col \left( e_n(i) \right) \text{,} &
	U_i &= \diag \left( u_{n,i} \right) \text{,} &
	\calM_i &= \diag \left( H_{n,i} \right) \text{,}
	\\
	a_i &= \col \left( a_n(i) \right) \text{,} &
	G &= C \kron I_M \text{,} &
	\calM_a &= \diag \left( \mu_{a,n} \right) \text{,}
	\\
	p_i &= \col \left( p_{n}(i) \right) \text{,} &
	\bar{\nu} &= \diag \left( \nu_n \right)  & \xi_i&=\colone_N\kron q_i.
\end{align*}
A set of equations describing the evolution of the entire network can then be obtained as follows.

Let $c_n^T$ be the $n$-th row of the combining matrix $C$ and $G_n = c_n^T \kron I_M$~(the $n$-th block-row of $G$). \add{Under \ref{A:etacheck}}, the equations for the overall network can be written in terms of the error vectors as
\begin{align}
	\Lambda_i &= \diag\left( \frac{1}{1 + e^{-a_n(i-1)}} \right) \text{,} \quad \calL_i = \Lambda_i \kron I_M \text{,}
		\label{eq:Lambdai}
	\\
	\tilw_i &= \calL_i \Psitl + \left( I_{MN} - \calL_i \right) G \tilw + \xi_i\text{,}
		\label{eq:wtli}
%	\tilw_i &= \calL_i \Psitl_{i-1} + \left( I_{MN} - \calL_i \right) G \tilw_{i-1} + \xi_i\text{,}\label{eq:wtli}
	\\
	\Psitl_{i} &= \begin{cases} 	\tilw, \qquad\qquad\qquad\text{if } \delta_{L_n}(i)=1,\\
	\left( I_{MN}-U_i^{T} \calM_i U_i \right) \Psitl - U_i^{T} \calM_i v_i + \xi_i,\  \text{otherwise}
.\end{cases}\raisetag{24pt}
		\label{eq:Psitli}
	\\
	e_i &= U_i \tilw_i + v_i - U_i\xi_i\text{,}
		\label{eq:ei}
	\\
	p_i &= \bar{\nu} p_{i-1} + \left(
			I_N - \bar{\nu} \right) \col\left( \abs{u_{n,i} \left( G_n \tilw - \Psitl_{n} \right)}^2
			\right) \text{,}\raisetag{17pt}
		\label{eq:Pi}
	\\
	\calM_{a,i} &= \calM_{a} \diag\left( \frac{1}{p_n(i) + \epsilon_p} \right) \text{,}
		\label{eq:mubarai}
	\\
	a_i = &\Big[ a + {\cal M}_{a,i} \Lambda_i \left(I - \Lambda_i \right) %\vphantom{\sum}\right. \times \left.
	\diag\left( u_{n,i} \left( G_n \tilw - \Psitl_{n} \right) \right)^{T} e_i		\Big],
		\label{eq:Ai}
\end{align}
where $\Psitl=\Psitl_{i-1}$, $\Psitl_n=\Psitl_{n,i-1}$ and $\tilw=\tilw_{i-1}$: some iteration indexes will be omitted in the sequel whenever necessary, and the brackets $[a_i]=[a_i]^{a_+}_{-a_+}$ in \eqref{eq:Ai} constrain each entry of vector $a_i$ to the interval $[-a_{+}, a_{+}]$. To obtain \eqref{eq:wtli}, we used the facts that $q_i$ is the same across the network, and that $\sum_{\ell\in\overline{\calN}_n} c_{n\ell} = 1$.

\subsection{Analysis in the mean}\label{ssec:M_Analysis}

Assume that the variance of $a_n(i-1)$ is small enough so that
\begin{equation}
	\Lambda_i \approx \Lambdabar_i = \diag\left( \lambdabar_n(i) \right)
		\eqdef \diag\left( \frac{1}{1+e^{-\E a_n(i-1)}} \right)
		\text{.}
\end{equation}
This assumption is reasonable if $\mu_a$ is small \add{(Assumption \ref{A:smallstepsize})}, or when the $\{a_n(i-1)\}$ are close to their limits at $\pm a_{+}$~\cite{Vitor09t}. \add{Under Assumption \ref{A:nu}}, we can approximate $p_i \approx \E \{p_i\}$, and
the evolution of $\abar_i \eqdef \E\{a_i\}$ is given by~(to simplify the following argument we introduce the auxiliary variable $\check{a}_i$ as below)
\begin{equation}\label{eq:Abar}
\begin{split}
	\check{a}_i &= \abar + \calMbar_{a} \Lambdabar \left(I - \Lambdabar\right) \E \left\{
		\diag \left( u_{n,i} \left( G_n \tilw - \Psitl_{n} \right)\right)^{T} e_i
	\right\} \text{,}
	\\
	\abar_i &= \left[ \check{a}_i \right]_{-a_{+}}^{a_{+}}
		\text{,}
\end{split}\raisetag{10pt}
\end{equation}
where $\Lambdabar=\Lambdabar_i$ and $\calMbar_{a}=\calMbar_{a,i} \eqdef \E \{ \calM_{a,i} \}$ will be evaluated further on.

Replacing
\begin{equation*}
	e_i = U_i \left( \calL_i \Psitl_{i-1} + \left(I_{MN} - \calL_i\right) G \tilw_{i-1} \right) + v_i - U_i \xi_i
\end{equation*}
into~\eqref{eq:Abar}, and recalling that $v_i$ and $\xi_i$ are independent of all other variables, we obtain
\begin{equation}\label{eq:Abarb}
\begin{aligned}
	\check{a}_i &= \abar + \calMbar_{a} \Lambdabar \left(I - \Lambdabar\right)\E \left\{
	\col \left[
		\left( G_n \tilw - \Psitl_{n} \right)^{T}\right.\right. \\
		%{}&\times u_{n,i}^{T} u_{n,i} \right.\right.\left.\left.
		{}&\left.\left.\times u_{n,i}^{T} u_{n,i}
		\left( \lambdabar_n(i) \Psitl_{n,i-1} + \left(1 - \lambdabar_n(i)\right) G_n \tilw \right)	\right]	\right\}
	\\
	{}&= \abar + \calMbar_{a} \Lambdabar \left(I - \Lambdabar\right) \col\left[
		\Tr \left( R_n \left( \lambdabar_n(i) \E\{ \Psitl_{n} \tilw^{T} \} G_n^T
		\right.\right.\right.
		\\
		{}&+ \left.\left.\left.
		\left(1 - \lambdabar_n(i)\right) G_n \E\{ \tilw \tilw^{T} \} G_n^T
		\right.\right.\right.- \left.\left.\left.
		\lambdabar_n(i) \E\{ \Psitl_{n} \Psitl_{n}^{T} \}
		\right.\right.\right.
		\\
		{}&- \left.\left.\left.
		\left(1 - \lambdabar_n(i)\right) G_n \E\{ \tilw\Psitl_{n}^{T} \}
	\right) \right) \right]
	\text{.}
\end{aligned}
\end{equation}
We see that for the evaluation of $\abar_i$ we must find
\begin{align}
	T_i &\eqdef \E\{\tilw_i \tilw_i^{T}\} \text{,} &
	S_i &\eqdef \E\{\tilw_i \Psitl_i^{T}\} \text{,} &
	K_i &\eqdef \E\{\Psitl_i \Psitl_i^{T}\}
		\text{,}
\end{align}
from which all expected values in \eqref{eq:Abar} can be obtained directly as follows. Note that
$\Psitl_{n,i-1} = \left(b_{(n)}^T \kron I_M\right) \Psitl_{i-1}$, where $b_{(n)}$ is the $n$-th $N\times 1$ canonical basis vector.  Denoting $B_n = b_{(n)}^T \kron I_M$, we can rewrite \eqref{eq:Abar} as
\begin{equation}\label{eq:Abari}
\begin{split}
	\abar_i &= \abar + \calMbar_a \Lambdabar \left(I - \Lambdabar\right) \col\left[
		\Tr \left( R_n \left( \lambdabar_n B_n S^{T} G_n^T\right.\right.\right. +\\
		{}&\left(1-\lambdabar_n\right) G_n T G_n^T-
		\lambdabar_n B_n K B_n^T-
		\left.\left.\left.\left(1-\lambdabar_n\right) G_n S B_n^T
	\right) \right) \right]
		\text{,}
\end{split}\raisetag{27pt}
\end{equation}
with $\lambdabar_n=\lambdabar_n(i)$, $K=K_{i-1}$, $S=S_
{i-1}$ and $T=T_{i-1}$.
Note that a recursion for $\calMbar_{a,i}$ can be also obtained from $T_{i-1}$, $S_{i-1}$ and $K_{i-1}$ as
\begin{equation}\label{eq:Emubarai}
	\calMbar_{a,i} = \E\{ \calM_{a,i} \} \approx \calM_a \diag \left( \frac{1}{\E\{p_n(i)\}+\epsilon_p} \right)
		\text{,}
\end{equation}
with $\pbar_n(i) \eqdef \E \{p_n(i)\}$ and
\begin{equation}\label{eq:Epki}
\begin{split}
	\pbar_n(i) & = \nu_n \pbar_n(i-1)
+ (1 - \nu_n) \E \left\{ \abs{u_{n,i} \left( G_n \tilw - \Psitl_{n} \right)}^2 \right\}
	\\
	{}&= \nu_n \pbar_n(i-1) + (1 - \nu_n) \Tr \left\{
		R_n \left( G_n T_{i-1} G_n^{T}
		\right.\right.
		\\
		{}&- \left.\left.
		G_n S_{i-1} B_n^T - B_n S_{i-1}^{T} G_n^{T} + B_n K_{i-1} B_n^T \right)
	\right\}
		\text{.}
\end{split}\raisetag{12pt}
\end{equation}

The covariance matrices $T_i$, $S_i$ and $K_i$ can be obtained from the autocorrelation matrix of the vector
$\Theta_i = \col\left( \tilw_i, \Psitl_i \right)$, \add{and will be dealt with in Section \ref{ssec:MS-Analysis}. A recursion for $\Theta_i$ can be obtained from \eqref{eq:Lambdai}--\eqref{eq:Ai} as follows.}
\begin{equation}\label{eq:Thetai}
	\Theta_i =
		\begin{bmatrix}
			(I_{MN}-\calL_i)G  &  \calL_i \\
			0              &  I - U_i^{*} \calM_i U_i
		\end{bmatrix} \Theta
		+ \begin{bmatrix}
			\xi_i \\
			\xi_i-U_i^{T} \calM_i v_i
		\end{bmatrix}
		\text{,}
\end{equation}
with $\Theta=\Theta_{i-1}$. Using again the assumption that the variance of $a_n(i)$ is small \add{(\ref{A:smallstepsize})}, the mean of \eqref{eq:Thetai} becomes
\begin{equation}\label{eq:Thetabari}
	\Thetabar_i \eqdef \E\{ \Theta_i \} =
	\underbrace{
		\begin{bmatrix}
			(I_{MN} - \bar{\calL}_i) G & \bar{\calL}_i \\
			0                    & I - R_{\mu}
		\end{bmatrix}}_{\eqdef F_i} \Thetabar_{i-1}
		\text{,}
\end{equation}
where $\bar{\calL}_i = \Lambdabar_i \kron I_M$, and $R_{\mu}$ depends on the particular algorithm used at each node. Assuming all nodes use LMS, we have $R_\mu = \diag\left(\mu_n R_n\right)$.  If the nodes use NLMS and the regressors $u_{n,i}$ are tap-delay lines, we can use the approximation $R_\mu \approx \diag\left(\frac{\mu_n}{M\sigma_{u,n}^2}R_n\right)$, where $\sigma_{u,n}^2=E\{u_n^2(i)\}$ \cite{tarrabConvergencePerformanceAnalysis1988}.  Of course, other algorithms can also be used with their appropriate models, and the algorithms need not be equal for all nodes.

We next study the stability and convergence of \eqref{eq:Thetabari}\footnote{This recursion is \emph{not} linear, since $\bar{\cal L}_i$ depends on the autocorrelation (not the autocovariance) of $\Theta_i$, and therefore depends also on its mean, $\Thetabar_i$.}.
Note first that \add{without transfer of coefficients (i.e., for $L\rightarrow\infty$)} the recursion for $\E\{\Psitl_i\}$ is linear and uncoupled from that for $\E\{\tilw_i\}$, so
$\E\{\Psitl_i\}$ converges to zero if and only if $\rho(I - R_{\mu}) < 1$, where $\rho(\cdot)$ denotes the spectral radius, or equivalently if $0 < \mu_n R_n < 2I$~\cite{Sayed08a}. The stepsizes must then satisfy
\begin{align}\label{eq:LMS.mean.stability}
	0 &< \mu_n < \begin{cases}
	2/\lambda_{\max}(R_n), &
		\text{for LMS, or}\\
		2, &\text{for NLMS}.\end{cases}
\end{align}
where $\lambda_{\max}(R_n)$ represents the largest eigenvalue of $R_n$.  Therefore, for small enough stepsizes we can guarantee that $\E\{\Psitl_i\} \rightarrow 0$. This should be no surprise, since the adaptive filters in each node are running independently of each other.

To prove stability and convergence \add{in the mean} of the entire scheme, \add{note that $F_i$ in \eqref{eq:Thetabari} is block-diagonal, and we just saw that under \eqref{eq:LMS.mean.stability} the lower diagonal block corresponds to a stable recursion.  We therefore must now show that the spectral radius of the upper diagonal block, i.e., $(I_{MN}-\bar{\cal L}_i)G$ is less than one.  This can be accomplished by showing that there is an induced norm such that $\norm{(I_{MN}-\bar{\cal L}_i)G}<1$ (since any induced norm upper bounds the spectral radius of a matrix \cite{Horn13}).}

\add{For that} we use the block-maximum norm~\cite{Sayed14a}. The block-maximum norm of a length-$MN$ vector is defined as follows: partition a length-$MN$ vector into length-$M$ blocks as $x = \col
(x_1, \dots, x_{N}) \in \mathbb{R}^{MN}$. Then the block maximum norm is defined as
\begin{equation}\label{eq:vector.b.norm}
	\norm{x}_{b,\infty} \eqdef \max_{1 \le n \le N} \norm{x_n}
		\text{,}
\end{equation}
where \add{from now on }$\norm{\cdot}$ \add{denotes} the Euclidean norm. For block matrices an induced norm based on \eqref{eq:vector.b.norm} is defined the usual way: let $A \in \mathbb{R}^{MN \times MN}$ be also partitioned into $M\times M$ blocks $A_{n,\ell}$ and define \cite{Sayed14a}
\begin{equation}\label{eq:matrix.b.norm}
	\norm{A}_{b,\infty} \eqdef \max_{\norm{x}_{b,\infty} \le 1} \norm{A x}_{b,\infty}
		\text{.}
\end{equation}

Let us now evaluate \add{$\norm{(I_{MN}-\bar{\cal L}_i)G}_{b,\infty}$}:
From the definitions of $\bar{\calL}_i = \Lambdabar_i \kron I_M$ and $G = C \kron I_M$\add{ and for any vector $x\in\re^{MN}$ with $\norm{x}_{b,\infty}\le 1$, we have}
\begin{align}\label{eq:block.Fix}
	(I_{MN} - \bar{\calL}_i) G \begin{bmatrix} x_1 \\ \vdots \\ x_N \end{bmatrix}
%		+ \alpha \bar{\calL}_i \begin{bmatrix} x_{N+1} \\ \vdots \\ x_{2N} \end{bmatrix} =
	%\notag\\
%	{}&
 &= \left[ C \kron I_M - (\Lambdabar_i C) \kron I_M \right]
		\begin{bmatrix} x_1 \\ \vdots \\ x_N \end{bmatrix}
%		+ \alpha (\Lambdabar_i \kron I_M)
%			\begin{bmatrix} x_{N+1} \\ \vdots \\ x_{2N} \end{bmatrix}
\notag\\
	{}&= \begin{bmatrix}
		\sum_{\ell =1}^N (1 - \lambdabar_{1}(i)) c_{1,\ell} x_{\ell}
			%+ \alpha \lambdabar_{1}(i) x_{n+N}
		\\
		\vdots
		\\
		\sum_{\ell =1}^N (1 - \lambdabar_{N}(i)) c_{N,\ell} x_{\ell}
			%+ \alpha \lambdabar_{N}(i) x_{n+N}
	\end{bmatrix}
	\notag\\
	{}& =
 \col \left\{
		(1 - \lambdabar_{n}(i)) \sum_{\ell =1}^N c_{n,\ell} x_{\ell}
			%+ \alpha \lambdabar_{n}(i) x_{n+N}
	\right\}
		\text{.}\raisetag{40pt}
\end{align}
Note that each one of the blocks in \eqref{eq:block.Fix} satisfies
\begin{align*}
	&\left\|
		(1 - \lambdabar_{n}(i)) \sum_{\ell =1}^N c_{n,\ell} x_{\ell}
			%+ \alpha \lambdabar_{n}(i) x_{n+N}
	\right\|
	\le \abs{1 - \lambdabar_n(i)} \sum_{\ell=1}^N c_{n,\ell} \norm{x_{\ell}}
		%+ \alpha \lambdabar_n(i) \norm{x_{n+N}}
	\\
	&\qquad {}\le (1 - \lambdabar_n(i)) \sum_{\ell=1}^Nc_{n,\ell} %+ \alpha \lambdabar_n(i)
		= 1 - %(1 - \alpha)
  \lambdabar_n(i)
		\text{,}
\end{align*}
where we used the facts that
(a)~$\norm{x}_{b,\infty} \le 1 \Rightarrow \norm{x_{\ell}} \le 1$, $1 \le \ell \le N$;
(b)~$0 < \lambdabar_n(i) < 1$, so $\abs{1 - \lambdabar_n(i)} = 1 - \lambdabar_n(i)$;
(c)~$\sum_{\ell=1}^N c_{n,\ell} = 1$.

Recalling that the $a_n(i)$ are restricted to the interval $[-a_{+}, a_{+}]$, the $\lambdabar_n(i)$ will stay in the interval
\begin{equation*}
	\lambdabar_n(i) \in \left[ \frac{1}{1+e^{a_{+}}}, \frac{1}{1+e^{-a_{+}}} \right]
		\text{,}
\end{equation*}
and %if we choose $\alpha < 1$,
\begin{equation*}
	0<1 - %(1 - \alpha)
 \lambdabar_n(i) \le 1 -
 %(1 - \alpha)
 \frac{1}{1+e^{a_{+}}} \eqdef \eta < 1
		\text{.}
\end{equation*}
We conclude that \add{$\norm{(I-\bar{\cal L}_i)Gx}_{b,\infty} \le \eta < 1$ for all $x\in \re^{MN}$ with $\norm{x}_{b,\infty}\le 1$.}  This means that \add{$\rho((I-\bar{\cal L}_i)G) \le 1-\eta < 1$, and thus } $\Thetabar_i$ converges exponentially fast to the origin, that is, the proposed scheme converges in the mean whenever the step-sizes $\mu_n$ are chosen so that all node filters are stable.

\subsection{Analysis in the mean square}\label{ssec:MS-Analysis}

Multiplying \eqref{eq:Thetai} by its transpose and taking expectations, we obtain recursions for $T_i$, $S_i$ and $K_i$.  The recursion for $K_i$ is a standard result for LMS or NLMS filters, since the local filters are operating independently of each other.  We restrict ourselves to the cases of LMS and NLMS only to keep the argument more concise --- if some nodes use different algorithms (such as RLS), the corresponding models from the literature can be substituted \cite{Sayed08a,Diniz13a}. Defining $\calQ=\E\xi_i\xi_i^T=Q\kron\colone\colone^T$, for LMS we have
\begin{equation}\label{eq:Ki.LMS}
	K_i = \begin{cases} \begin{aligned} &K - R_{\mu} K - K R_{\mu} + \diag\left(\trace(\mu_n R_n K)R_n\right) \\
	&+ 2R_\mu K R_\mu+ \diag\left(\mu_n^2 \sigma_{v,n}^2 R_n\right) + \calQ, \text{ if }\delta_L(i)=1
	\end{aligned} \\
	T_{i-1}, \qquad\qquad \text{if } \delta_L(i)=0.\end{cases}\raisetag{8pt}
\end{equation}
whereas an approximate model for NLMS is the recursion  \cite{tarrabConvergencePerformanceAnalysis1988}
\begin{equation}\label{eq:Ki.NLMS}
	K_i = \begin{cases} %\begin{aligned}
	K ~-~ R_{\mu} K ~-~ K R_{\mu}~ +~ %\diag\left(\trace\left(\frac{\mu_n}{M\sigma_{u,n}^2} R_n K_{n,i-1}\right)R_n\right)\\
	%&+ 2
	R_\mu K R_\mu & \\
	+ ~\diag\left(\frac{\mu_n^2}{M(M-2)\sigma_{u,n}^4} \sigma_{v,n}^2 R_n\right) + \calQ,
	%\end{aligned}
	& \text{ if }\delta_L(i)=1\\
	T_{i-1}, \qquad\qquad \text{ if }\delta_L(i)=0\end{cases}%\raisetag{60pt}
\end{equation}
recalling that here we denote $K=K_{i-1}.$

For $T_i$ and $S_i$, we obtain
\begin{align}
	T_i &= \left(I_{MN} - \bar{\calL}\right) G T G^T \left(I_{MN} - \bar{\calL}\right)
		+ \bar{\calL} S^{T} G^T \left(I_{MN} - \bar{\calL}\right) \notag\\
		&+ \left(I_{MN} - \bar{\calL}\right) G S \bar{\calL}
	+ \bar{\calL} K \bar{\calL} + \calQ
			\text{,}
	\label{eq:Ti}\\
	S_i &= \begin{cases} \begin{aligned}
	    &\left(I_{MN} - \bar{\calL}\right) G S \left(I_{MN} - R_{\mu}\right)\\
		&+ \bar{\calL} K \left(I_{MN} - R_{\mu}\right) + \calQ
		\end{aligned} \qquad \qquad\text{ if }\delta_L(i)=1\\
		T_{i-1} \qquad\qquad \text{ if }\delta_L(i)=0.
	\end{cases}\label{eq:Si}
\end{align}
where $\bar{\calL}=\bar{\calL}_i$.

A model for the overall algorithm is obtained running \eqref{eq:Ki.LMS}--\eqref{eq:Si} and \eqref{eq:Abari}--\eqref{eq:Epki} sequentially. Given the highly nonlinear nature of the problem, we leave a stability analysis of the recursion for a future work.

\section{Simulations}\label{S:Sims}

In this section we study the three adaptive combiners described earlier: the MSD-alg algorithm, given by \eqref{E:MSDSupervisor}; the LS-alg algorithm, defined in \eqref{E:LSAdaptiveSupervisor}; and the proposed U-sup algorithm, collected in \eqref{E:UniversalSupervisor}.

For ease of reference, we repeat some definitions here, and introduce others. The input signals $\{u_n(i)\}$ are zero-mean Gaussian sequences generated according to
\begin{equation}\label{eq:Corr_model}
    u_n(i) = \beta_n u_n(i-1) +\sqrt{1-\beta_n^2} x_n(i)\text{,}
\end{equation}
in which $x_n(i)$ is a Gaussian i.i.d. zero-mean signal with unit variance $\sigma^2_{x,n}=1$, and $-1<\beta_n<1$ is the correlation factor. The measured signal $d_n(i)$ is generated according to the data model (\ref{eq:data-model-tv}). The measurement noise $v_n(i)$ is a Gaussian i.i.d. sequence whose variance $\sigma^2_{v,n}$ is adjusted at the nodes to achieve the SNR profiles, randomly selected, that are presented in each simulation example.

Both stationary and non-stationary scenarios are considered, so that the $M\times 1$ time-varying unknown plant $w^o_i$ follows the random walk model defined in (\ref{eq:RandomWalk}),
with initial condition $w^o_{-1}=\frac{1}{\sqrt{M}}\colone$, and $q_i$ is a zero-mean, i.i.d. Gaussian vector process, with covariance matrix $Q=\sigma^2_q I_M$.  For stationary plants, $\sigma_q^2=0$ and $w^o_i=w^o$. The adaptive network in charge of tracking such a plant is comprised of $N=15$ nodes equipped with local NLMS filters also with order $M=50$, and step-sizes randomly selected as either $\mu_{n}=0.1$, or $\mu_{n}=0.01$ (except for Example 5), with regularization $\epsilon=10^{-6}$.

We design scenarios so as to explore the desired properties for ANs discussed in Section \ref{ssec:Defs_Universality}, and to conclude on the  universality of the algorithms. The adopted metrics are the local $\MSD_n(i)$ and network $\MSD(i)$ mean-square deviations (check (\ref{eq:Local_Errors}) and (\ref{eq:Network_Errors})), defined in terms of $\psi_{n,i-1}$ for the non-cooperative case, in terms of $\phi_{n,i-1}$ for the MSD-alg case, and in terms of $w_{n,i}$ for the U-sup and LS-alg. We only present a few U-sup combiners $\{\text{E}~\lambda_n(i)\}$, in order to promote picture clarity; information on combiners for the other algorithms may be obtained directly in \cite{Takahashi10d} and \cite{Bes2017}.

There are three different error figures: the transient curves, represented by $\MSD(i)$ as a function of the iterations; the steady-state curves, given by $\MSD_n(\infty)$ versus the node index; and a robustness curve, which depicts global MSD quantities versus the non-stationarity parameter $\sigma_q^2$, obtained as follows. The minimum local MSD across the network, $\min_n \MSD_n(\infty)$, is considered for the non-cooperative case; the distributed algorithms are represented by their maximum local MSD, i.e., $\max_n \MSD_n(\infty)$; in other words, the worst node in each algorithm must be equal or better than the best non-cooperative node.
Such pictures show how sensitive the distributed algorithms are with respect to how rapidly the plant evolves. Notice that universality may be inferred from both the steady-state and the robustness MSD curves.

The algorithms are compared in fair scenarios, with their parameters optimized for the set of studied cases. Namely, U-sup uses
$\mu_{a,n}=\mu_a=0.005$ (in Example 4, $\mu_a=5\cdot 10^{-4}$) and feedback cycle $L_n=L$, with either $L=800$, or \add{no transfer of coefficients ($L\to\infty$)}, $\nu_n = 0.9$ and $\epsilon_p = 0.01$. The LS-alg uses $\eps_{LS}=10^{-8}$ and $\gamma_n=\gamma=0.9999$; in Example 1 we add one LS-alg curve with $\gamma=0.99$ for comparison, as suggested by the authors \cite{Bes2017}.\footnote{We tested $\gamma\in\{0.99,0.999,0.9999\}$ in order to obtain the best LS-alg performance. For all the examples included here, $\gamma=0.9999$ resulted in the best error levels at steady-state, while $\gamma=0.99$ presented better transient, but with a considerable degradation after convergence. Different levels for $\eps_{LS}$ were also considered, although without relevant impact on performance.} The MSD-alg step-size $\mu_M$ is implemented with $\kappa_M=10^{-5}$ (this choice resulted in better performance than the value $\kappa_M=0.8$ used in \cite{Takahashi10d}) and $\eps_M=10^{-3}$ (See (14) and (18) in \cite{Takahashi10d}).

%=======================================
%==== Example 1 - BadNode ==============
%=======================================

In \textbf{Example 1}, we test the ability of the algorithms in rejecting a low-SNR at a well connected node. Node $1$ has SNR = $-4.9~dB$ and a strict node degree $\bar{\calN}_1=8$ (the most connected). The $N$ inputs $\{u_n(i)\}$ are white, the unknown vector $w^o$ is stationary at first
and the SNR and the stepsizes across the nodes are
SNR=$[-4.9, 11.8, 19.1, 15.7, 16.4, 14.5, 13.8, 15.9, 11.7, 12.1,$ $11.6, 11.1, 18.9, 14.6, 18.1]$ and $\mu_k=0.1\cdot[1, 10, 1, 1, 1,$ $10, 1, 10, 1, 10, 1, 1, 1, 10, 1, 1]$. The network topology and the adaptive combiner mean evolution $\text{E}\lambda_n(i)$ for our universal AN are presented in Figs. \ref{fig:Example1_top_lambdas}-(a) and (b). This is a worst case scenario for the AN:  a node subject to a high noise level will produce poor estimates, that will rapidly spread since Node 1 is well-connected. Fig. \ref{fig:Example1_MSD_MSD}-(a) shows the network MSD$(i)$ evolution for a stationary plant and Fig. \ref{fig:Example1_MSD_MSD}-(b) depicts the algorithms tracking a non-stationary plant following model (\ref{eq:RandomWalk}), with $\sigma^2_q=5\cdot10^{-6}$. Figure \ref{fig:Example1_MSD_MSD}-(a) shows how the LS-alg can perform very well in some cases where the plant is either stationary, or varies very slowly, and can be an option if its extra computational complexity is not a problem; the dark blue LS-alg curve uses $\gamma=0.99$, presenting a faster transient, while the blue LS-alg curve uses $\gamma = 0.9999$ and attains a better steady-state performance. However, when the plant varies faster, as happens in Fig. \ref{fig:Example1_MSD_MSD}-(b), the LS-alg cannot keep up, presenting a major drop in performance, while the U-sup and the MSD-alg perform much better. Figure \ref{fig:Example1_SS_Rob}-(a) presents the steady-state performance $\MSD_n(\infty)$ at the node level for $\sigma^2_q=5\cdot10^{-6}$, and one can see that both U-sup and MSD-alg are universal, while the LS-alg is not. Figure  \ref{fig:Example1_SS_Rob}-(b) depicts the robustness of all three methods: for slowly varying plants, up to a certain degree, the LS-alg performs better than the other two algorithms. However, as the plant starts evolving faster, LS-alg degrades more rapidly than the other algorithms, becoming worse than the non-cooperative case beyond $\sigma^2_q=5\cdot 10^{-8}$, therefore losing universality.  The U-sup is the only algorithm that outperforms the non-cooperative case across the entire $\sigma^2_q$ test range, also when $\sigma_q^2\to 0$ (tested, but not shown).

\begin{figure}[!tb]
	%\centering
	\begin{minipage}[c]{0.40\columnwidth}
	\centering
		\includegraphics[width=\columnwidth]{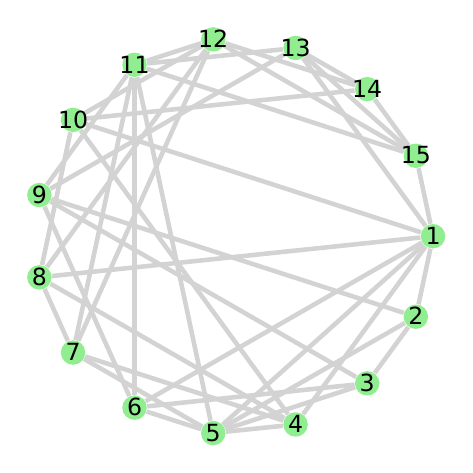}
		\small{(a)}
	\end{minipage}
	\hfill
	\begin{minipage}[c]{0.5\columnwidth}
	\centering
		\includegraphics[width=\columnwidth]{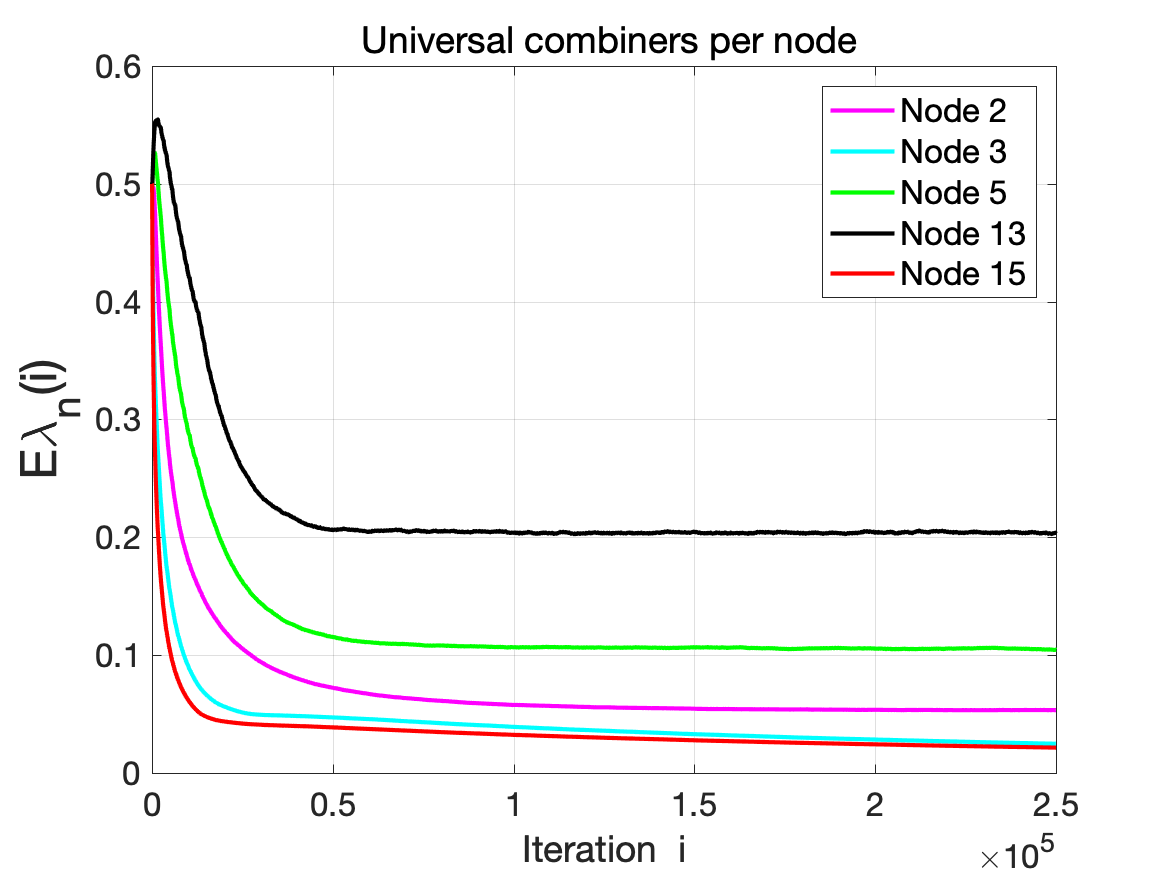}
		\vspace{-0.5cm}
		\small{(b)}
	\end{minipage}
	%\vspace*{0.2cm}
	\caption{\textbf{Example 1 (White inputs)}: (a) Network topology; (b) The mean adaptive combiners E$\lambda_{n}(i)$ corresponding to Fig. \ref{fig:Example1_MSD_MSD}-(a).}
		\label{fig:Example1_top_lambdas}
	\vspace{-0.2cm}
\end{figure}
%

%======= Transient: LS x U
\begin{figure}[!tb]
	\begin{minipage}[c]{0.49\columnwidth}
	    \centering
		\includegraphics[width=\columnwidth]{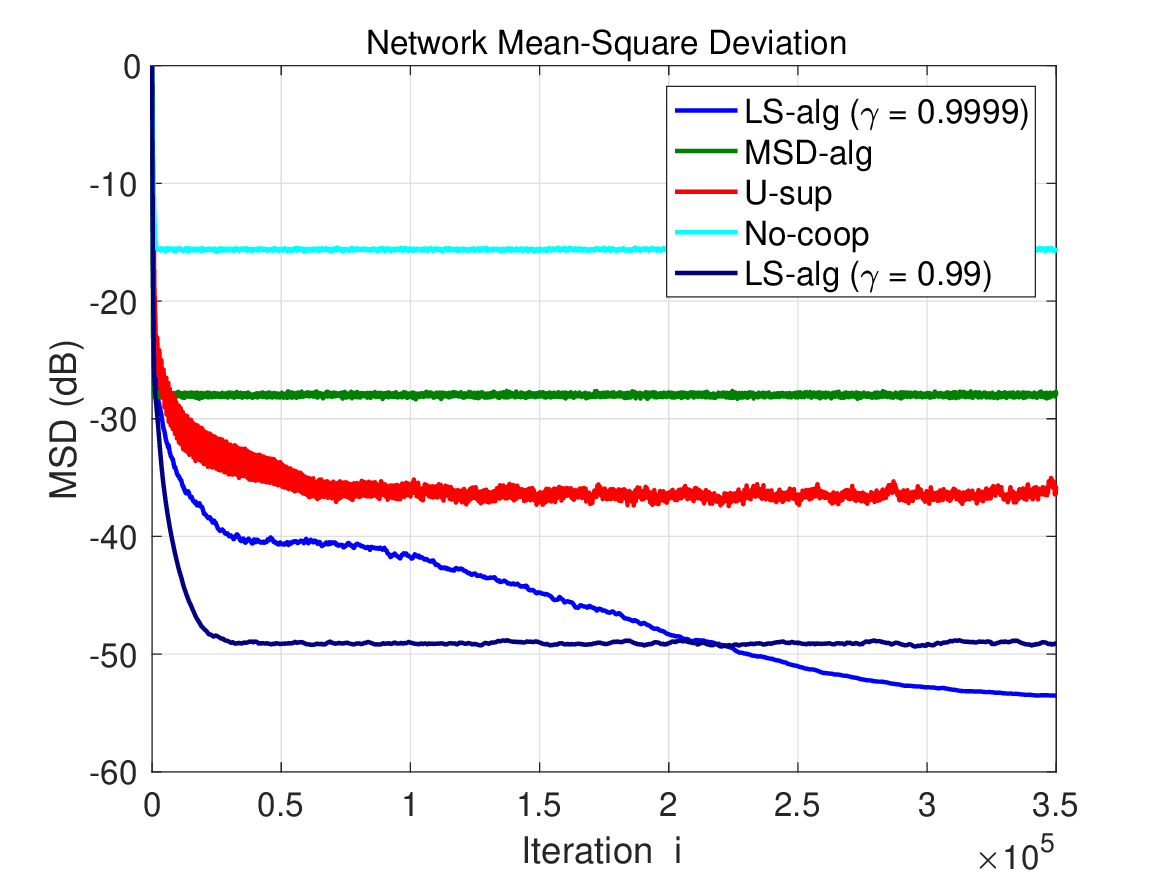}
		%BadNode_msd_Ex1.eps (anterior)
		\small{(a)}
	\end{minipage}
	\hfill
	\begin{minipage}[c]{0.49\columnwidth}
	    \centering
		\includegraphics[width=\columnwidth]{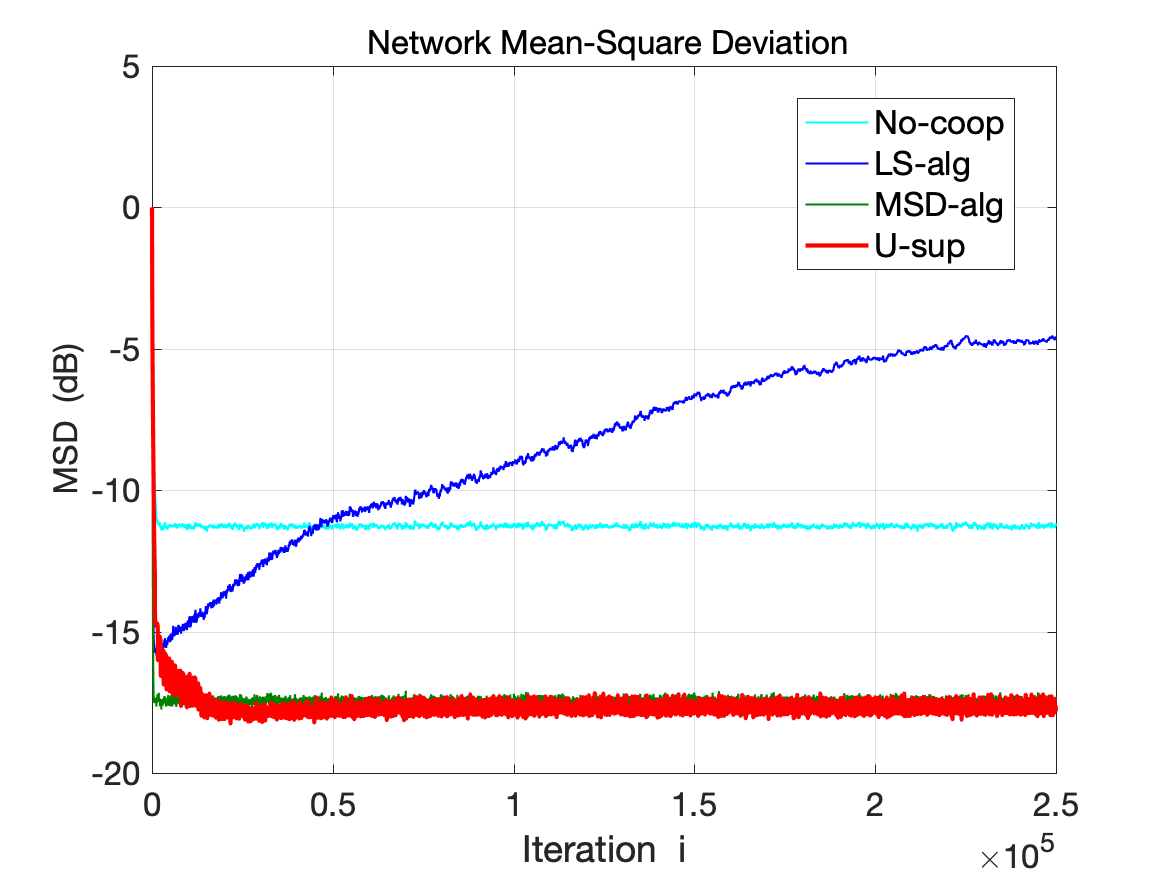}
		\small{(b)}
	\end{minipage}
	%\vspace*{0.2cm}
	\caption{\textbf{Example 1 (White inputs)}: (a) Network MSD$(i)$ for stationary plant; (b) Network MSD$(i)$ for a random walk plant with $\sigma^2_q=5\cdot10^{-6}$.}
		\label{fig:Example1_MSD_MSD}
	\vspace{-0.5cm}
\end{figure}

%==  Steady-state TV8  and  Robustness wrt average MSD
\begin{figure}[!tb]
	\begin{minipage}[c]{0.49\columnwidth}
	    \centering
		\includegraphics[width=\columnwidth]{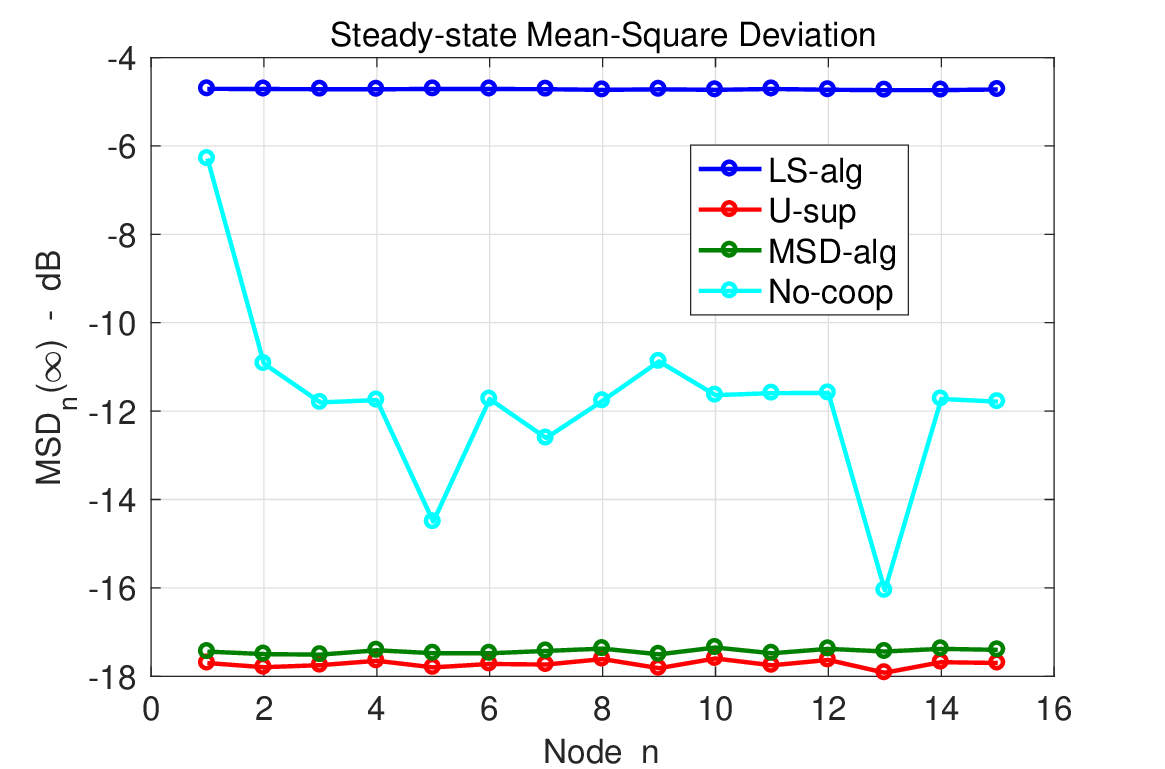}
		\small{(a)}
	\end{minipage}
	\hfill
	\begin{minipage}[c]{0.46\columnwidth}
	    \centering
		\includegraphics[width=\columnwidth]{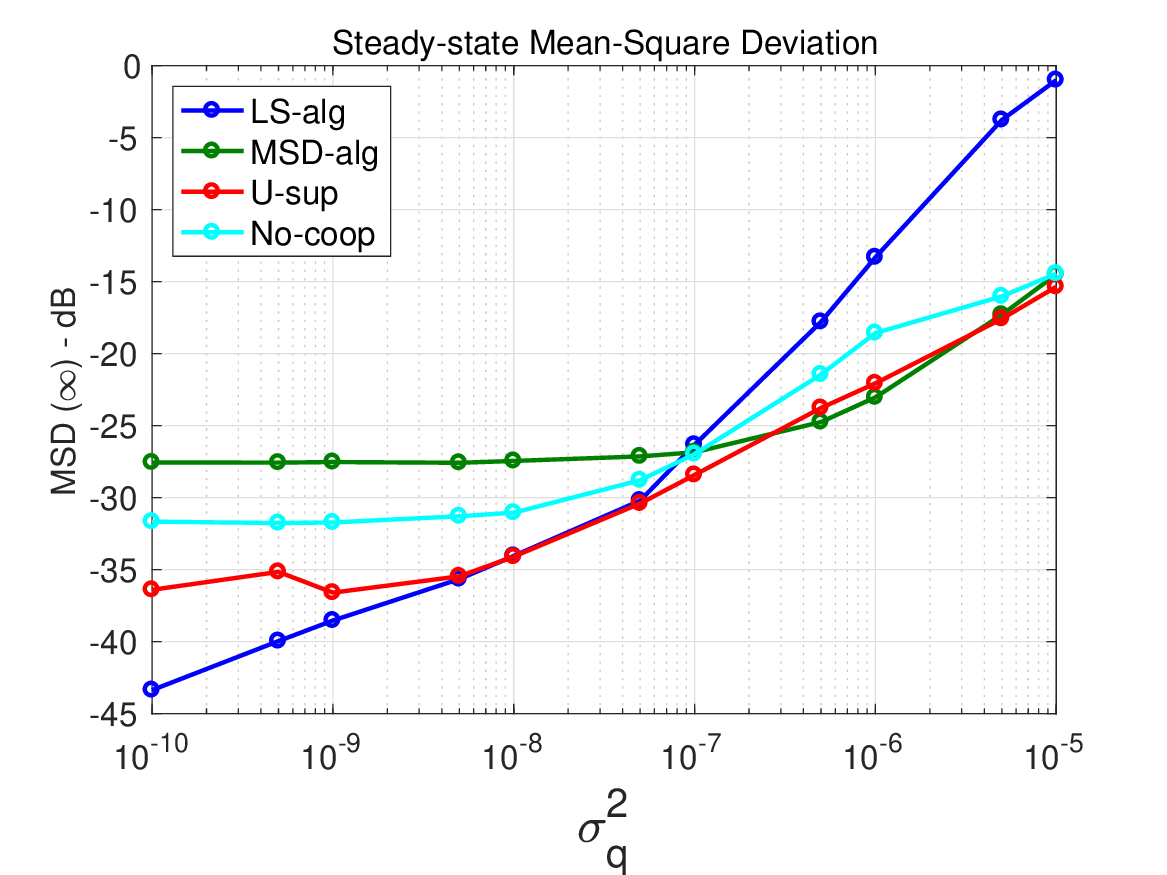}
		\small{(b)}
	\end{minipage}
	%\vspace*{0.2cm}
	\caption{\textbf{Example 1 (White inputs)}: (a) Steady-state MSD$_n(\infty)$ for $\sigma^2_q=5\cdot10^{-6}$ corresponding to Fig. \ref{fig:Example1_MSD_MSD}-(b); (b) Tracking robustness in terms of MSD$(\infty$) versus  $\sigma^2_q$: $\min_n$ MSD$_n(\infty)$ for the non-cooperative case and $\max_n$ MSD$_n(\infty)$ for the cooperative algorithms.}
		\label{fig:Example1_SS_Rob}
	\vspace{-0.5cm}
\end{figure}

%=======================================
%==== Example 2 - GoodNode =============
%=======================================

In \textbf{Example 2}, we test the network ability to recruit and propagate the exceptional Node 1, which is poorly connected (one connection only). This is the opposite scenario of Example 1. Figs \ref{fig:Example2_top_lambdas}-(a) and (b) respectively depict the network topology and the mean combiners E$\lambda_n(i)$ for the U-sup algorithm corresponding to the network MSD depicted in Fig.\ref{fig:Example2_MSD_MSD}-(a). The input data $\{u_n(i)\}$ is white and the plant $w^o_i$ again evolves according to (\ref{eq:RandomWalk}). The network SNR and step-sizes are respectively SNR=$[22.5, 11.6, 14.5, 17.7, 10.1, 14.7, 10.3, 18.7, 10.5, 17.8,$
$15.9, 12.4, 17.8, 15.2, 19.5]$ and $\mu_k=0.1\cdot[1, 10, 1, 1, 10, 1,$ $10, 1, 10, 1, 1, 1, 10, 1, 1]$.
Figure \ref{fig:Example2_MSD_MSD}-(a) shows how U-sup and LS-alg algorithms perform (much) better than the non-cooperative case, and better than the MSD-alg, when the time-varying plant evolves under $\sigma_q^2=10^{-7}$. Increasing the plant velocity to $\sigma_q^2=10^{-5}$, in Fig. \ref{fig:Example2_MSD_MSD}-(b), the LS-alg is nearly 10 dB worse (and still worsening: it has not converged after $2\cdot 10^5$ iterations) than the U-sup and MSD-alg, and worse than the non-cooperative case. Figure \ref{fig:Example2_SS_rob}-(a) corresponds to the steady-state for Fig. \ref{fig:Example2_MSD_MSD}-(a), and shows that both U-sup and LS-alg are universal for $\sigma_q^2=10^{-7}$, while the MSD-alg is not. The algorithm robustness is depicted by Fig. \ref{fig:Example2_SS_rob}-(b), again confirming that the U-sup is the only algorithm that is universal across the entire  $\sigma_q^2$ test range, also depicting the LS-alg sensitivity when tracking fast varying plants.

\begin{figure}[!t]
	\begin{minipage}[c]{0.40\columnwidth}
	    \centering
		\includegraphics[width=\columnwidth]{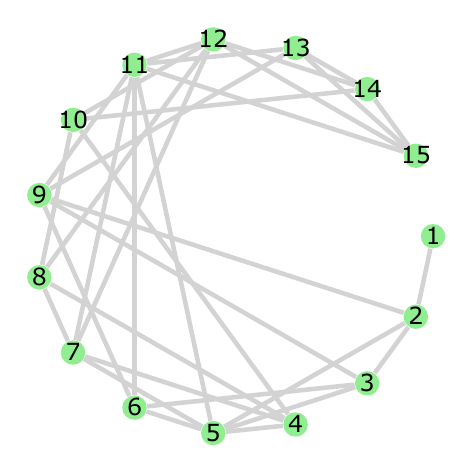}
		\small{(a)}
	\end{minipage}
	\hfill
	\begin{minipage}[c]{0.50\columnwidth}
	    \centering
		\includegraphics[width=\columnwidth]{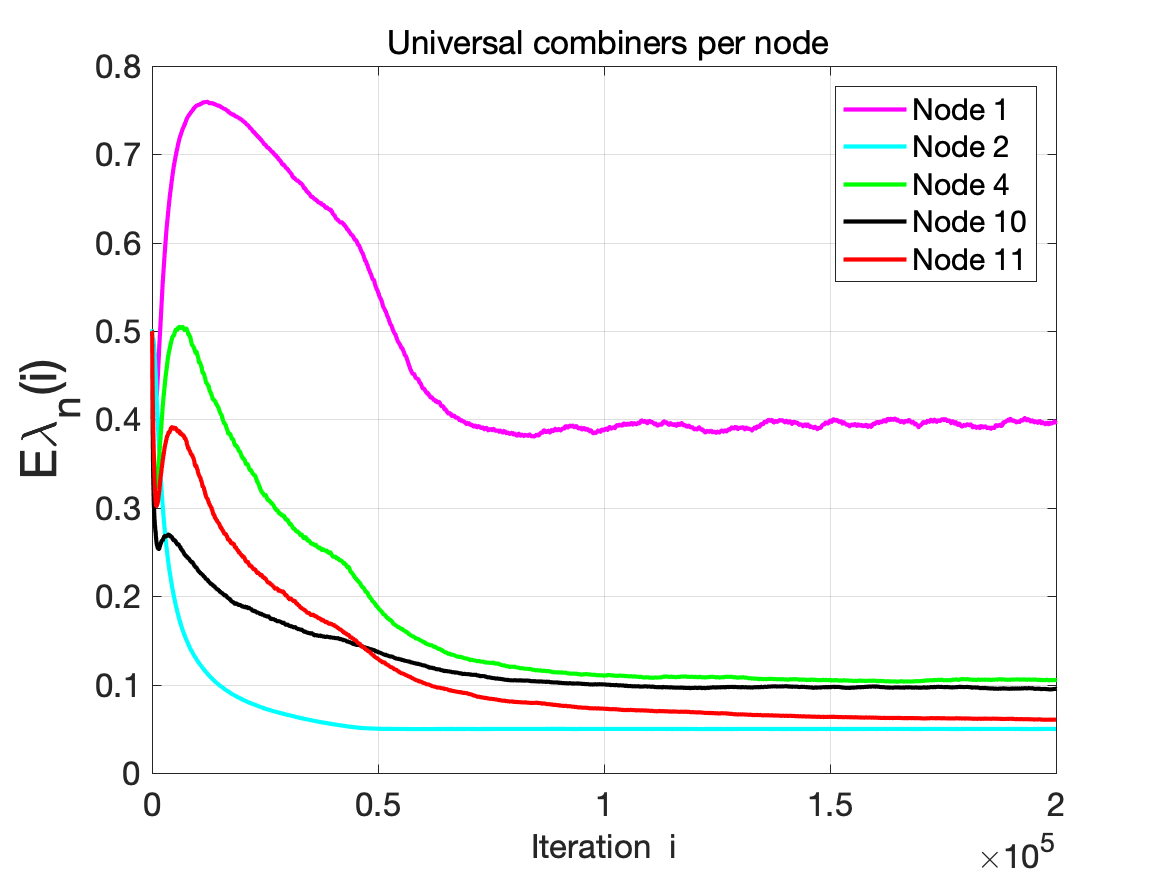}
		\small{(b)}
	\end{minipage}
	%\vspace*{0.2cm}
	\caption{\textbf{Example 2 (White inputs)}: (a) Network topology; (b) The mean adaptive combiners E$\lambda_{n}(i)$ corresponding to Fig.\ref{fig:Example2_MSD_MSD}-(a).}
		\label{fig:Example2_top_lambdas}
	\vspace{-0.5cm}
\end{figure}

\begin{figure}[!t]
	\begin{minipage}[c]{0.49\columnwidth}
	    \centering
		\includegraphics[width=\columnwidth]{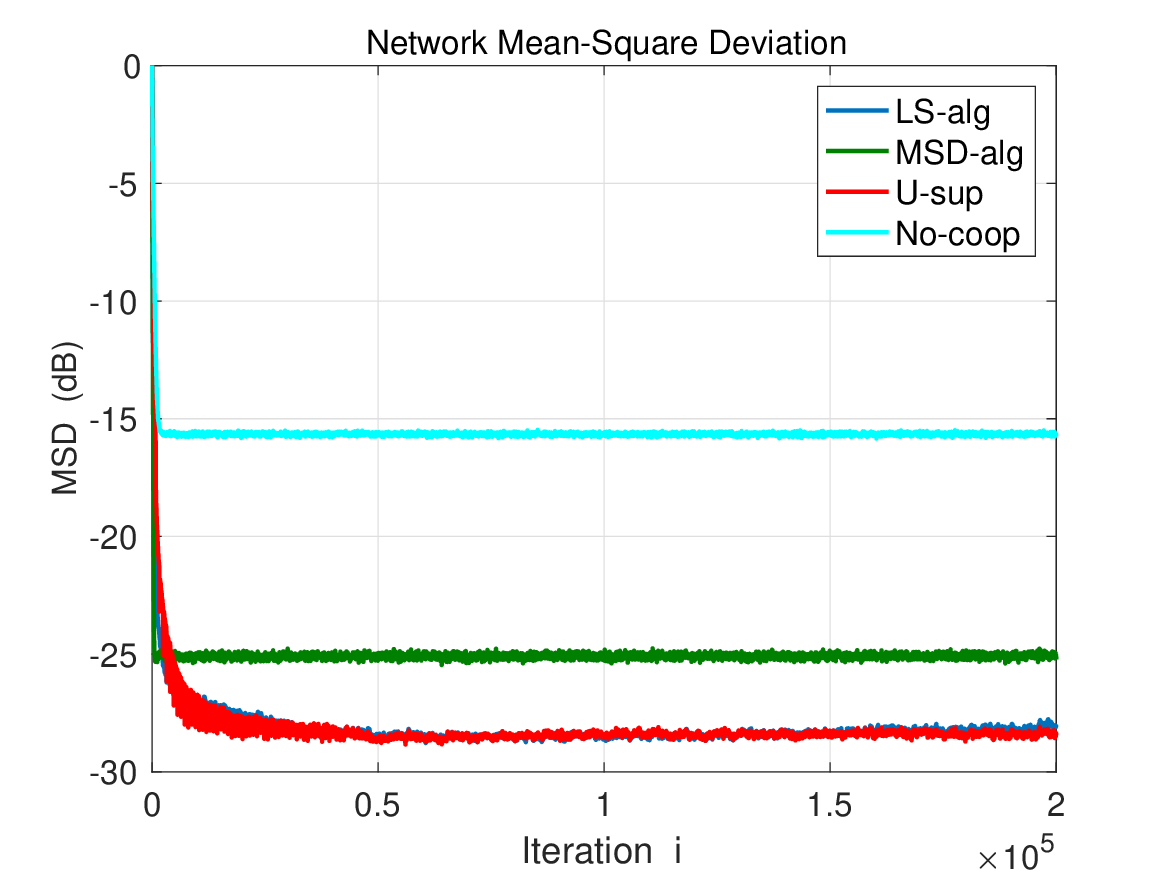}
		\small{(a)}
	\end{minipage}
	\hfill
	\begin{minipage}[c]{0.49\columnwidth}
	    \centering
		\includegraphics[width=\columnwidth]{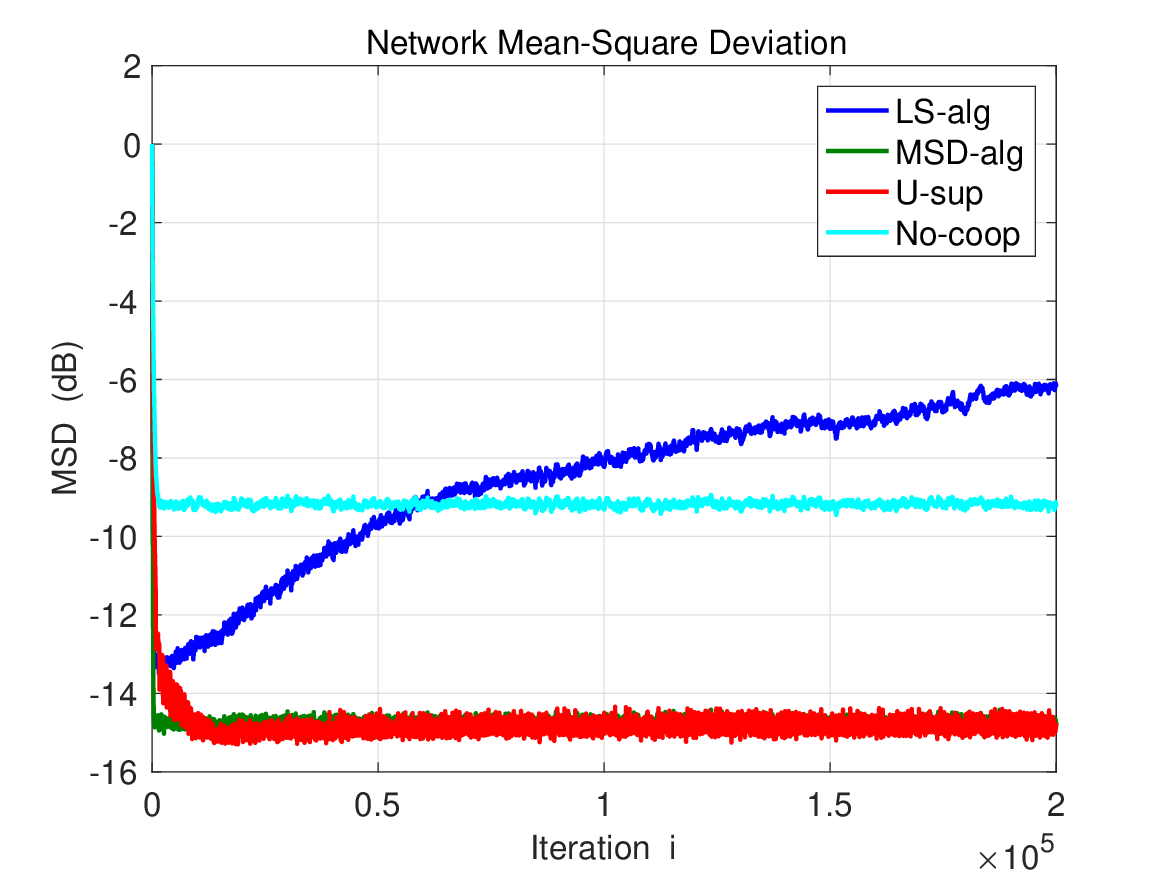}
		\small{(b)}
	\end{minipage}
	%\vspace*{0.2cm}
	\caption{\textbf{Example 2 (White inputs)}: (a) Network $\MSD(i)$ for a non-stationary plant with $\sigma^2_q=10^{-7}$; (b) Network $\MSD(i)$ when $\sigma^2_q=10^{-5}$.}
		\label{fig:Example2_MSD_MSD}
	\vspace{-0.3cm}
\end{figure}

\begin{figure}[!t]
	\begin{minipage}[c]{0.49\columnwidth}
	    \centering
		\includegraphics[width=\columnwidth]{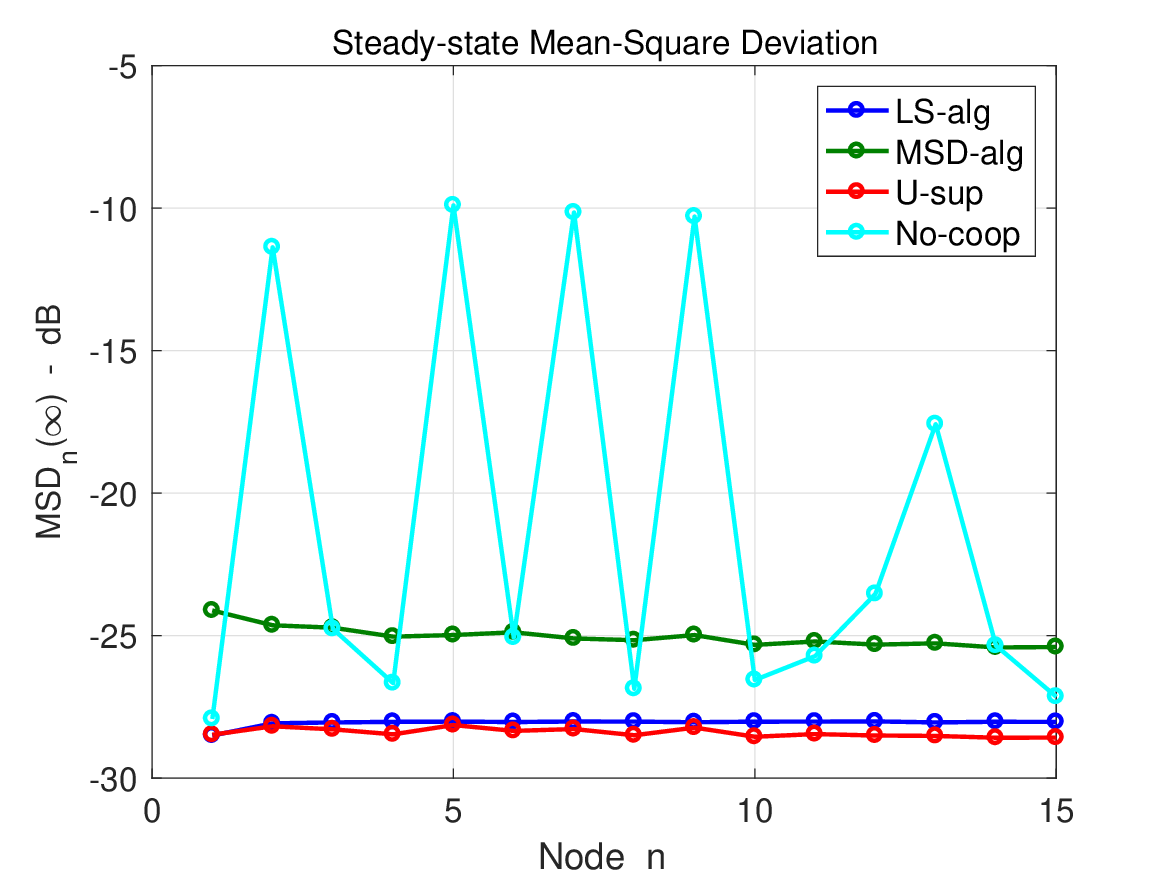}
		\small{(a)}
	\end{minipage}
	\hfill
	\begin{minipage}[c]{0.49\columnwidth}
	    \centering
		\includegraphics[width=\columnwidth]{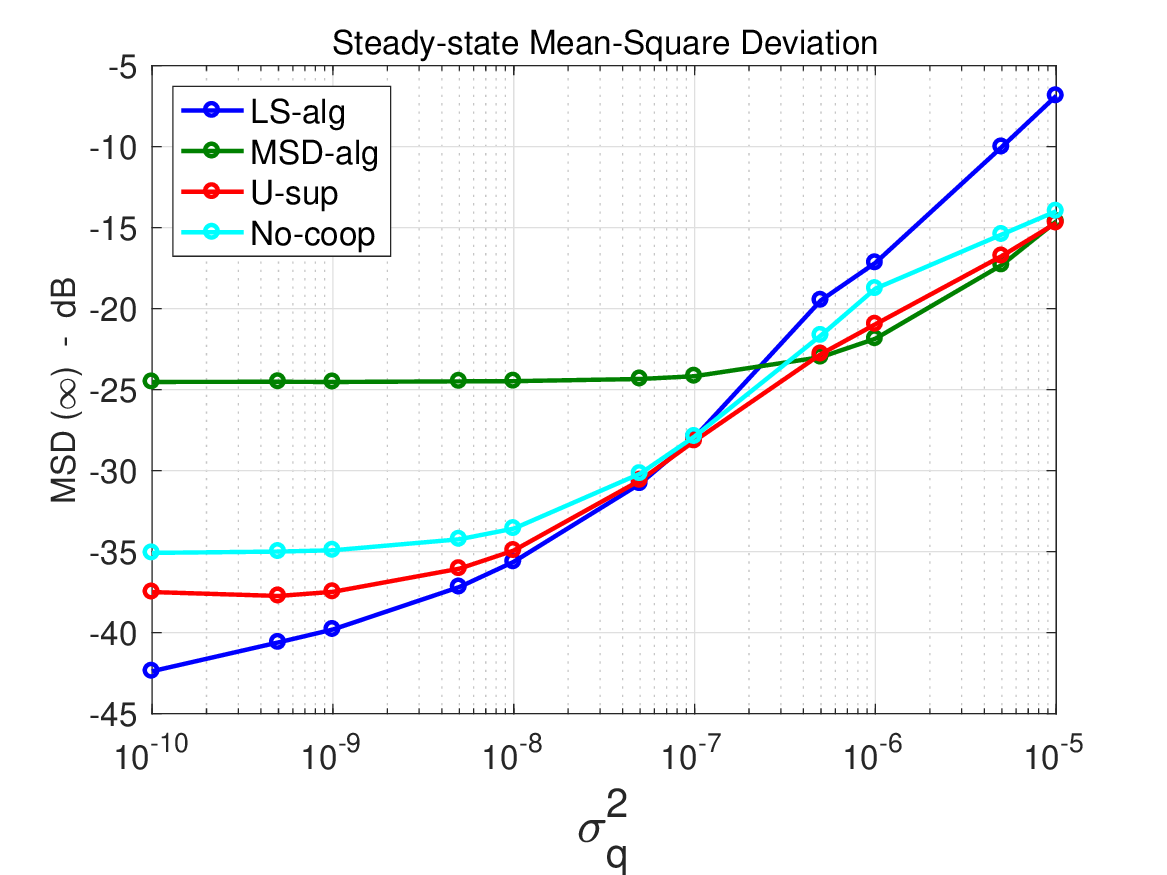}
		\small{(b)}
	\end{minipage}
%	\vspace*{0.2cm}
	\caption{\textbf{Example 2 (White inputs)}: (a) Steady-state $\MSD_n(\infty)$ for Fig. \ref{fig:Example2_MSD_MSD}-(a);  (b) Tracking robustness in terms of MSD$(\infty)$ versus $\sigma^2_q$: $\min_n$ MSD$_n(\infty)$ for the non-cooperative case and $\max_n$ MSD$_n(\infty)$ for the cooperative algorithms.}
		\label{fig:Example2_SS_rob}
	\vspace{-0.3cm}
\end{figure}

%=======================================
%==== Example 3 - Average Node =========
%=======================================

\textbf{In Example 3}, the scenario is what is more likely to take place in practice, where nodes have approximately the same node degree, without a clearly exceptional or poor node in terms of SNR, and considering correlated inputs $\{u_n(i)\}$ obtained from (\ref{eq:Corr_model}), with correlation factors $\{\beta_n\}$ randomly selected from the two values $0.63$ and $0.95$ and captured by the vector $\beta=[0.63, 0.95, 0.63, 0.63, 0.95, 0.63, 0.95, 0.63, 0.95, 0.63,$
$0.63, 0.63, 0.95, 0.63, 0.63]$. Figures \ref{fig:Example3_top_lambdas}-(a) and (b) depict network topology and the mean combiners E$\lambda_{n}(i)$ corresponding to Figs. \ref{fig:Example3_MSD_MSD}-(a) and \ref{fig:Example3_SS_Rob}-(a); the former presenting the performance for a stationary plant, the latter of a non-stationary plant. The network SNR and step-sizes are SNR$~=~[12.2, 15.2, 15.5, 15.5, 20, 11.2, 17.2, 12.4, 17.8, 16.1,$ $12.6, 10.1, 19.5, 12.1, 12.6]$ and $\mu_k=0.1\cdot[1, 10, 1, 1, 10, 1, 10,$ $1, 10, 1, 1, 1, 10, 1, 1]$.
We then  test the algorithms when tracking a non-stationary plant with $\sigma^2_q = 10^{-7}$ in Fig. \ref{fig:Example3_MSD_MSD}-(b).  Observe in Figs. \ref{fig:Example3_MSD_MSD}-(a) and (b) how the U-sup algorithm outperforms both the MSD-alg and LS-alg algorithms. As depicted in Figs. \ref{fig:Example3_SS_Rob}-(a) and (b), the three algorithms present a similar performance, however once more only the U-sup attains universality in the entire test range for $\sigma^2_q$.

\begin{figure}[!t]
	\begin{minipage}[c]{0.40\columnwidth}
	    \centering
		\includegraphics[width=\columnwidth]{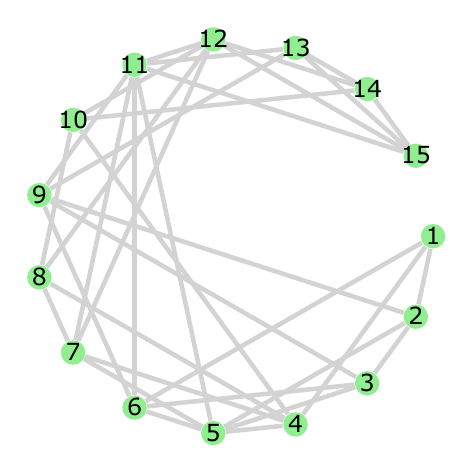}
		\small{(a)}
	\end{minipage}
	\hfill
	\begin{minipage}[c]{0.50\columnwidth}
	    \centering
		\includegraphics[width=\columnwidth]{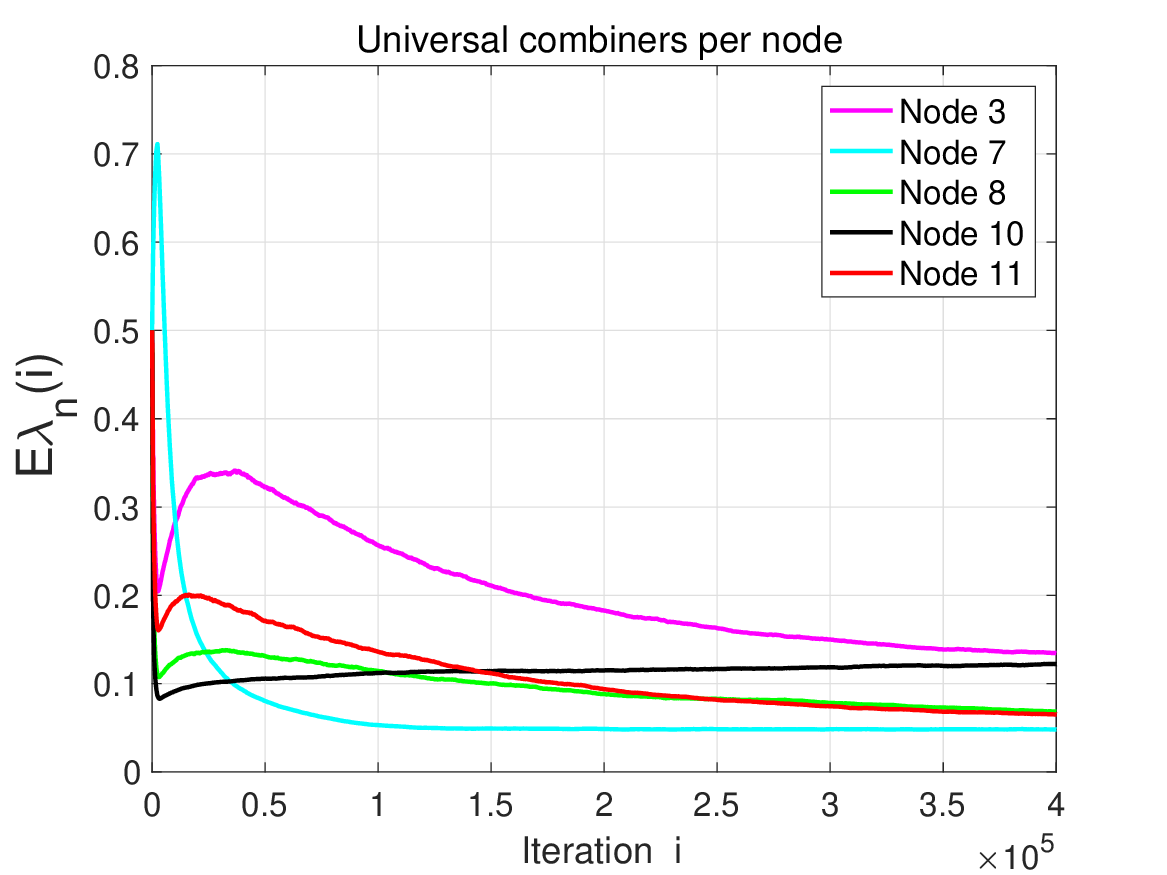}
		\small{(b)}
	\end{minipage}
%	\vspace*{0.2cm}
	\caption{\textbf{Example 3 (Correlated inputs)}: (a) Network topology; (b) The mean adaptive combiners E$\lambda_{n}(i)$ corresponding to Figs. \ref{fig:Example3_MSD_MSD}-(a) and \ref{fig:Example3_SS_Rob}-(a).}
		\label{fig:Example3_top_lambdas}
	\vspace{-0.3cm}
\end{figure}

\begin{figure}[!t]
	\begin{minipage}[c]{0.49\columnwidth}
	    \centering
		\includegraphics[width=\columnwidth]{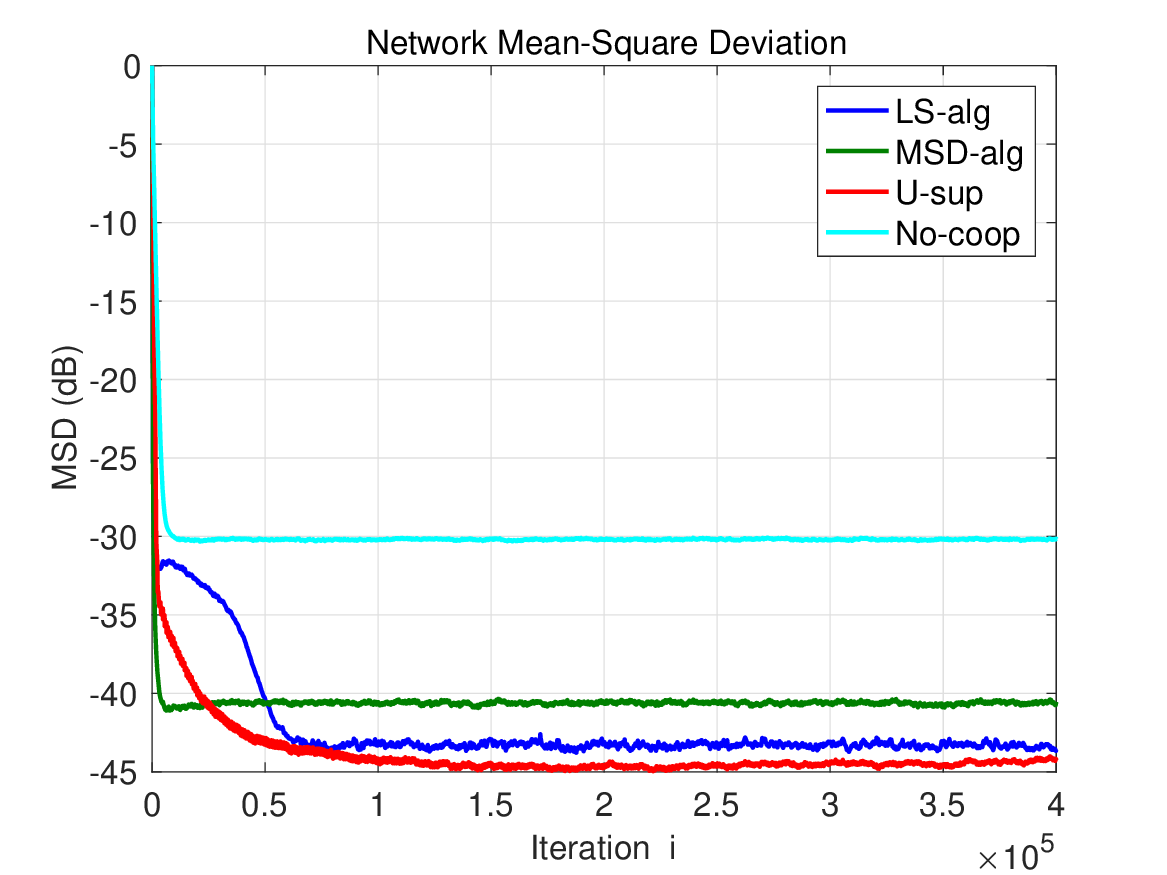}
		\small{(a)}
	\end{minipage}
	\hfill
	\begin{minipage}[c]{0.49\columnwidth}
	    \centering
		\includegraphics[width=\columnwidth]{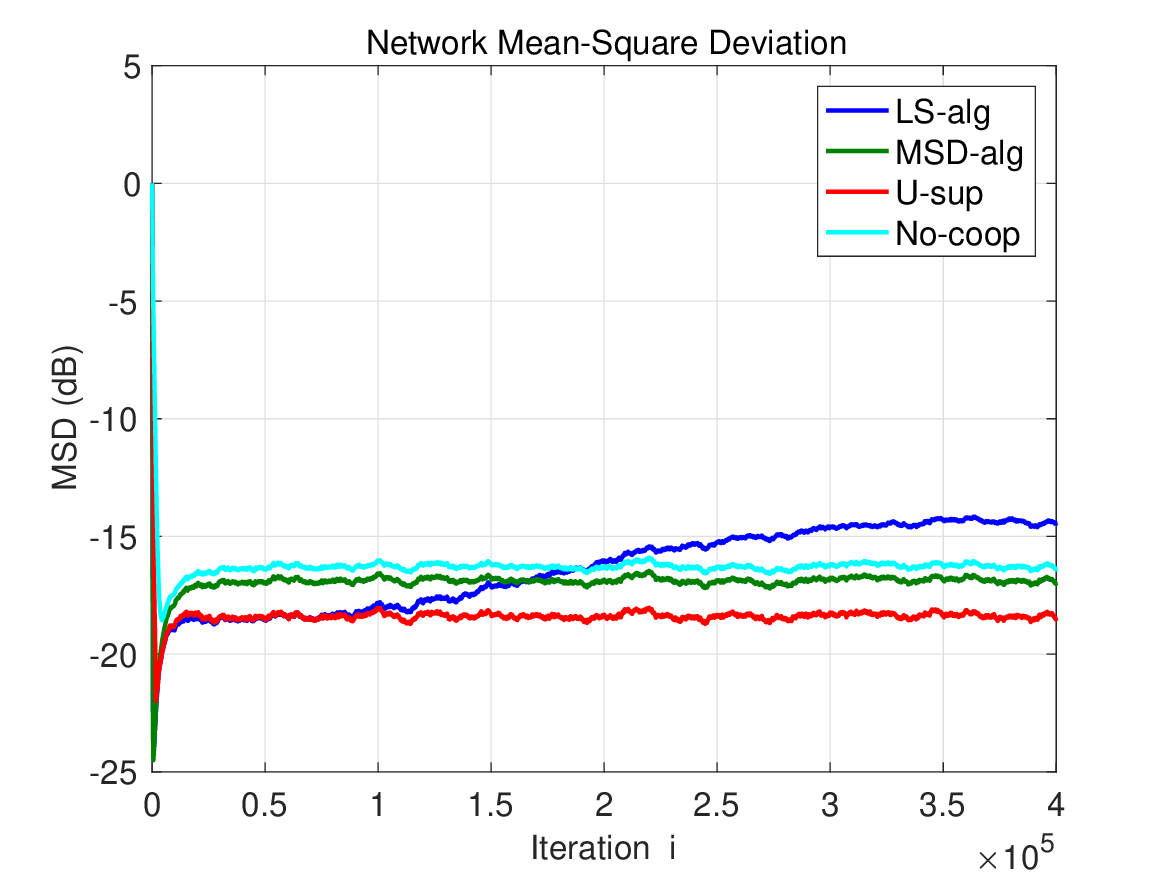}
		\small{(b)}
	\end{minipage}
	%\vspace*{0.2cm}
	\caption{\textbf{Example 3 (Correlated inputs)}: (a) Network $\MSD(i)$ for correlated inputs and a stationary plant; (b) Network $\MSD(i)$ for the same inputs and a time-varying plant with $\sigma^2_q = 10^{-7}$.}
		\label{fig:Example3_MSD_MSD}
	\vspace{-0.3cm}
\end{figure}

% == Non-Stationary ==========

\begin{figure}[!t]
	\begin{minipage}[c]{0.49\columnwidth}
	    \centering
		\includegraphics[width=\columnwidth]{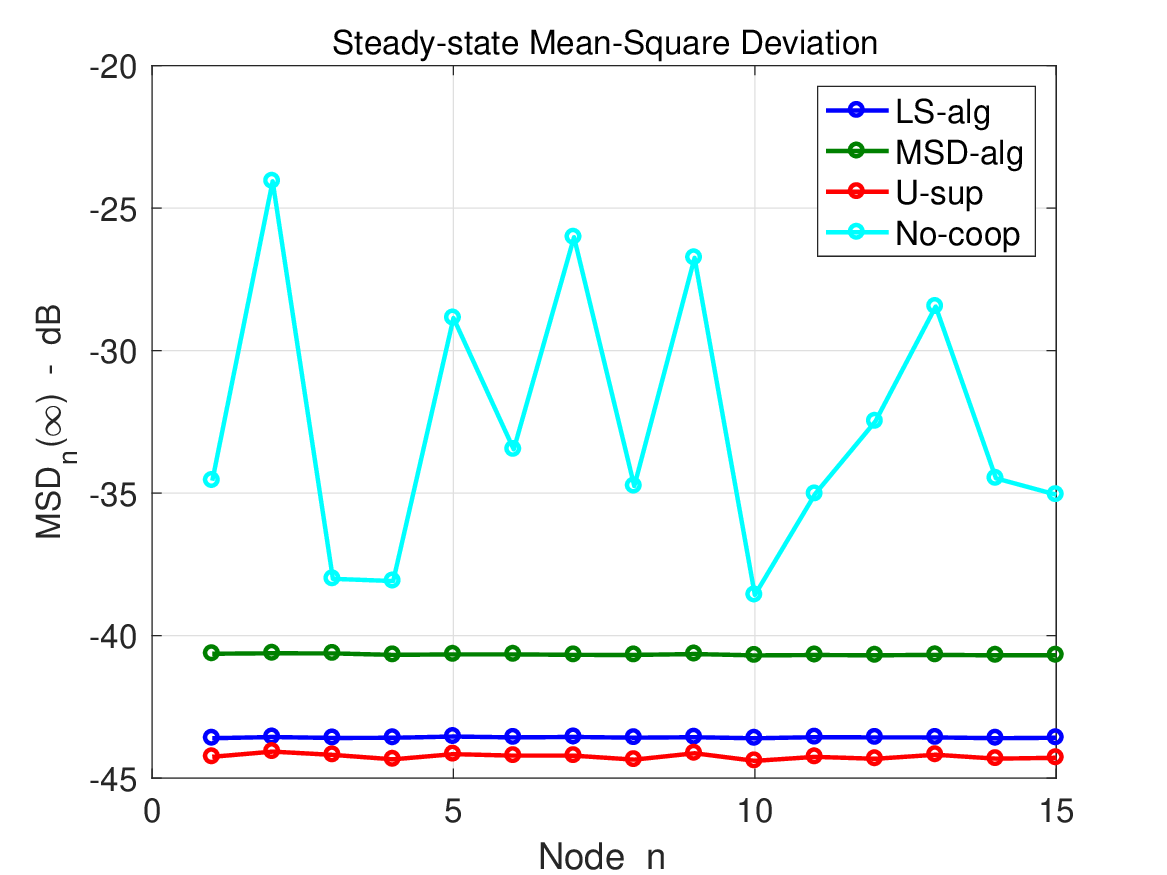}
		\small{(a)}
	\end{minipage}
	\hfill
	\begin{minipage}[c]{0.49\columnwidth}
	    \centering
		\includegraphics[width=\columnwidth]{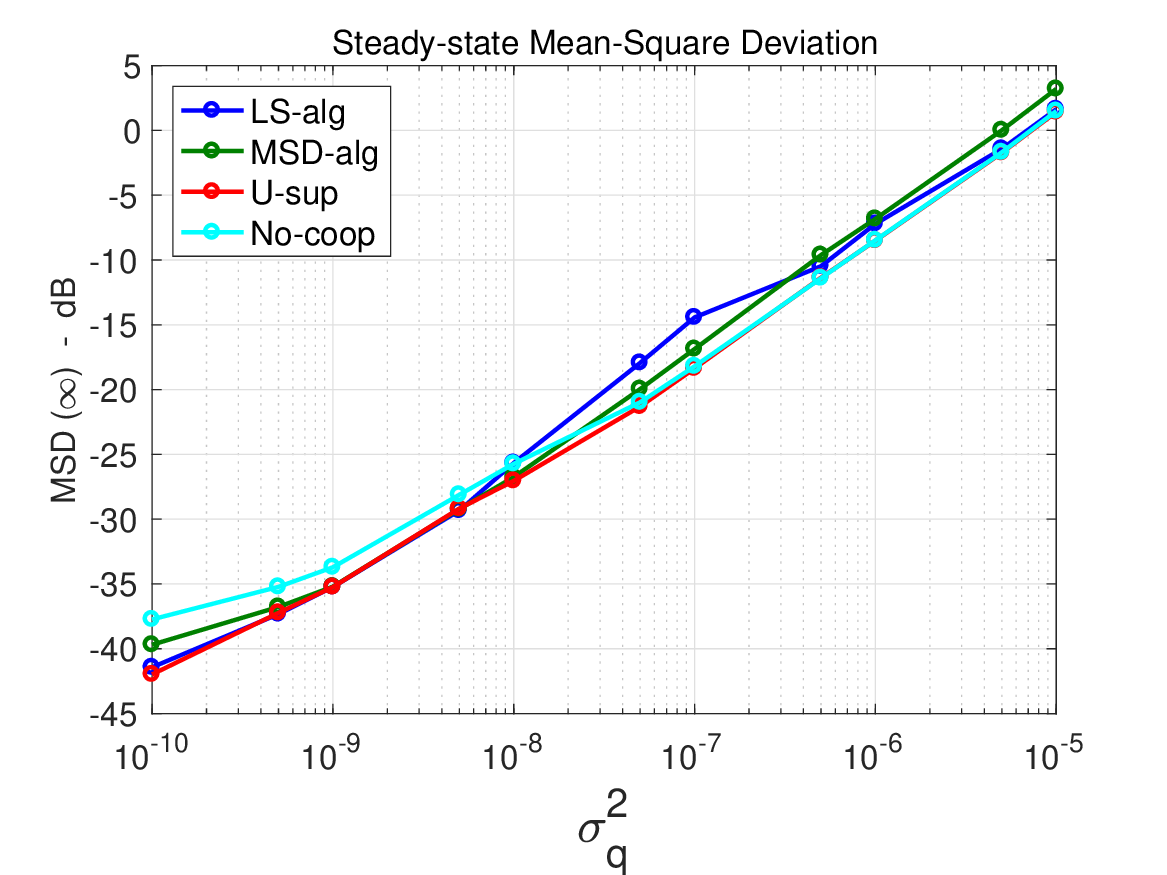}
		\small{(b)}
	\end{minipage}
%	\vspace*{0.2cm}
	\caption{\textbf{Example 3 (Correlated inputs)}: (a) Steady-state $\MSD_n(\infty)$;  (b) Tracking robustness in terms of MSD versus $\sigma^2_q$: $\min_n$ MSD$_n(\infty)$ for the non-cooperative case and $\max_n$ MSD$_n(\infty)$ for the cooperative algorithms.}
		\label{fig:Example3_SS_Rob}
	\vspace{-0.3cm}
\end{figure}

%=======================================
%==== Example 4 - Model =========
%=======================================

In \textbf{Example 4}, we test the theoretical model developed in Section \ref{sec:Analysis}. For that, we revisit the scenario from Example 3, changing the non-stationarity level for $\sigma^2_q=10^{-10}$, and set the feedback cycle $L\to\infty$.
Figure \ref{fig:Example4_MSD_eta}-(a) shows the network MSD evolution, and Fig. \ref{fig:Example4_MSD_eta}-(b) presents the corresponding mean combiner evolution $\text{E}\lambda_n(i)$. Note how the theoretical model is able to capture the general tendencies correctly.

\begin{figure}[!t]
	\begin{minipage}[c]{0.49\columnwidth}
	    \centering
		\includegraphics[width=\columnwidth]{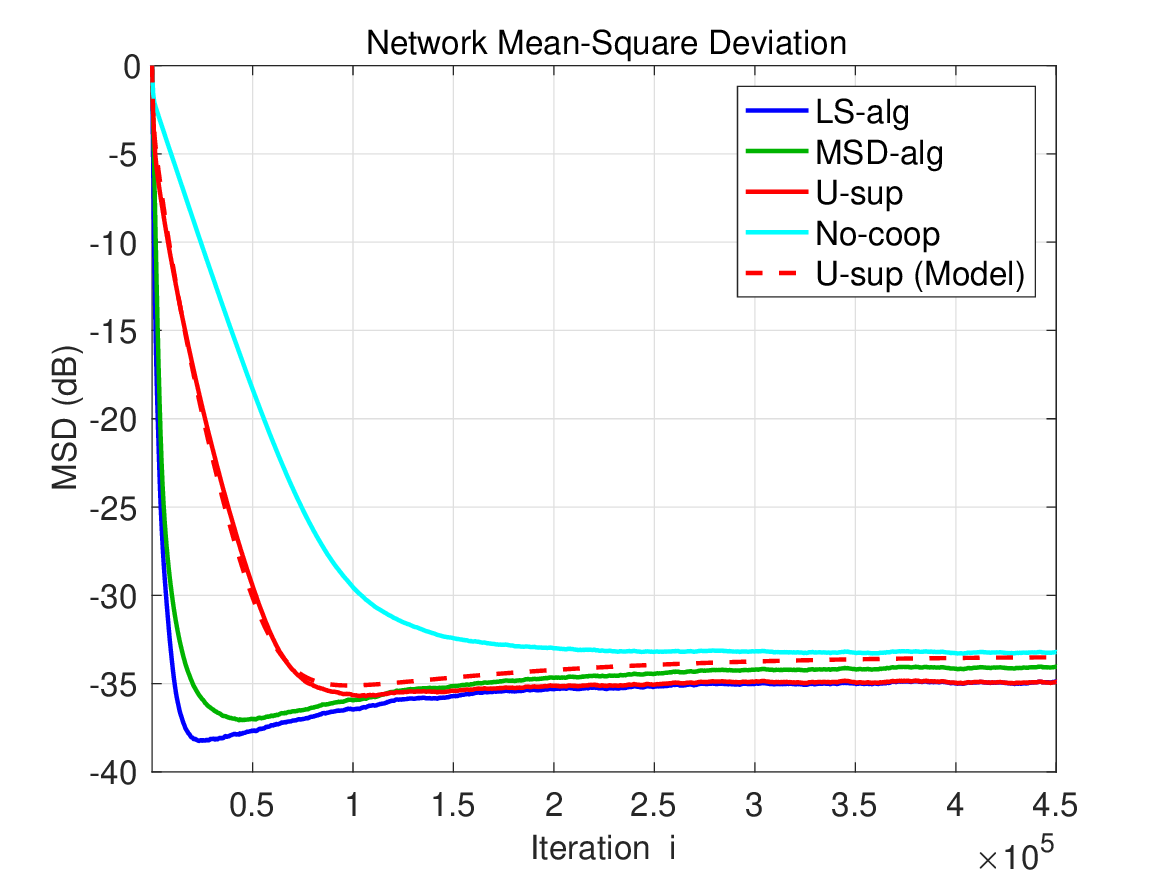}
		\small{(a)}
	\end{minipage}
	\hfill
	\begin{minipage}[c]{0.49\columnwidth}
	    \centering
		\includegraphics[width=\columnwidth]{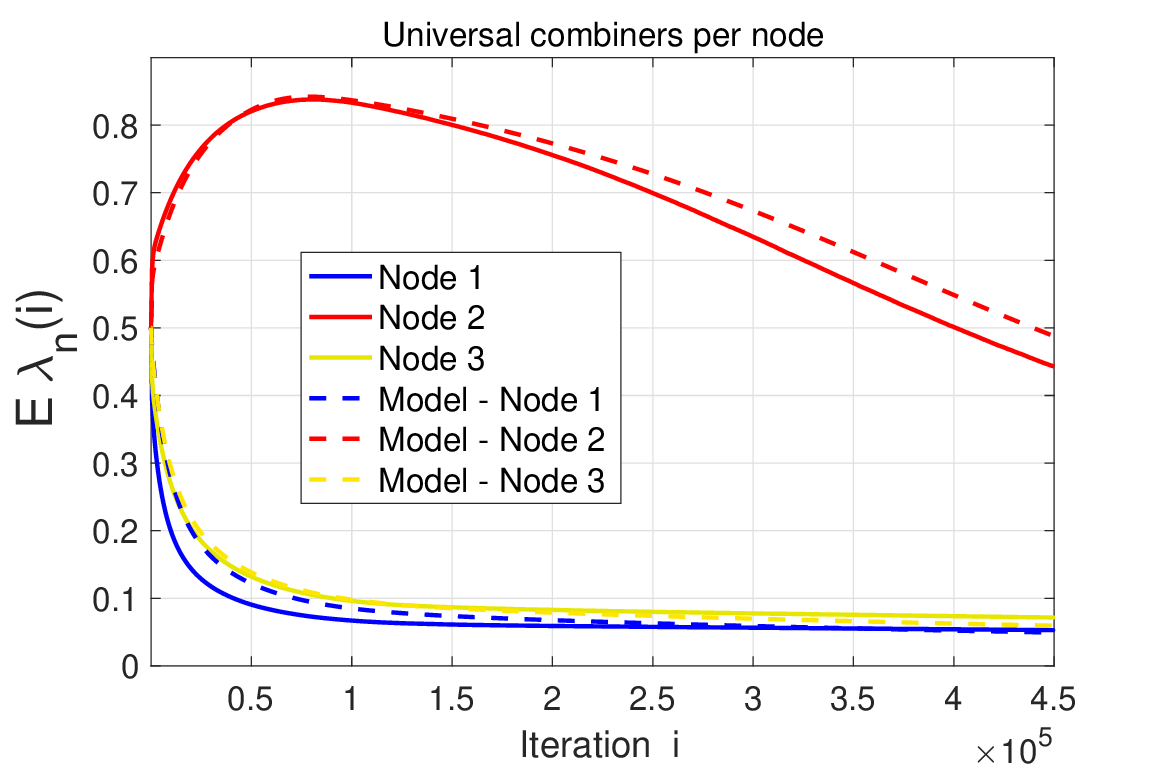}
		\small{(b)}
	\end{minipage}
%	\vspace*{0.2cm}
	\caption{\textbf{Example 4 (Correlated inputs)}: (a) Network $\MSD(i)$ and the theoretical model for U-sup (dashed red); (b) The mean adaptive combiners E$\lambda_{n}(i)$ for a few nodes, corresponding to Fig \ref{fig:Example4_MSD_eta}-(a).}
		\label{fig:Example4_MSD_eta}
	\vspace{-0.3cm}
\end{figure}

%=======================================
%==== Example 5 - Model =========
%=======================================

In \textbf{Example 5}, we test the theoretical model from Section \ref{sec:Analysis} for white inputs in a network with $N=8$, NLMS AFs with order $M=6$, with the U-sup  using $L=800$, all identifying a stationary plant. The stepsizes for this example are captured by the vector $\mu=0.01\cdot[1, 10, 1, 10, 1, 1, 1, 1]$. The network topology is described by the reduced undirected edge set $E'=[12, 13, 18, 24, 25, 34, 56, 58, 67, 78]$, in which, for example, the pair $12$ means there are edges between nodes $1$ and $2$, between node $1$ and itself and between $2$ and itself. The SNR across the network is SNR$~=~[11.2, 10.6, 18.4, 13.4, 17.8, 11.2, 16.8, 10.9]$. The theoretical model is shown in Figs. \ref{fig:Example5_MSD_eta}(a) and (b) to describe well the network MSD performance and the combiner evolution.

\begin{figure}[!t]
	\begin{minipage}[c]{0.49\columnwidth}
	    \centering
		\includegraphics[width=\columnwidth]{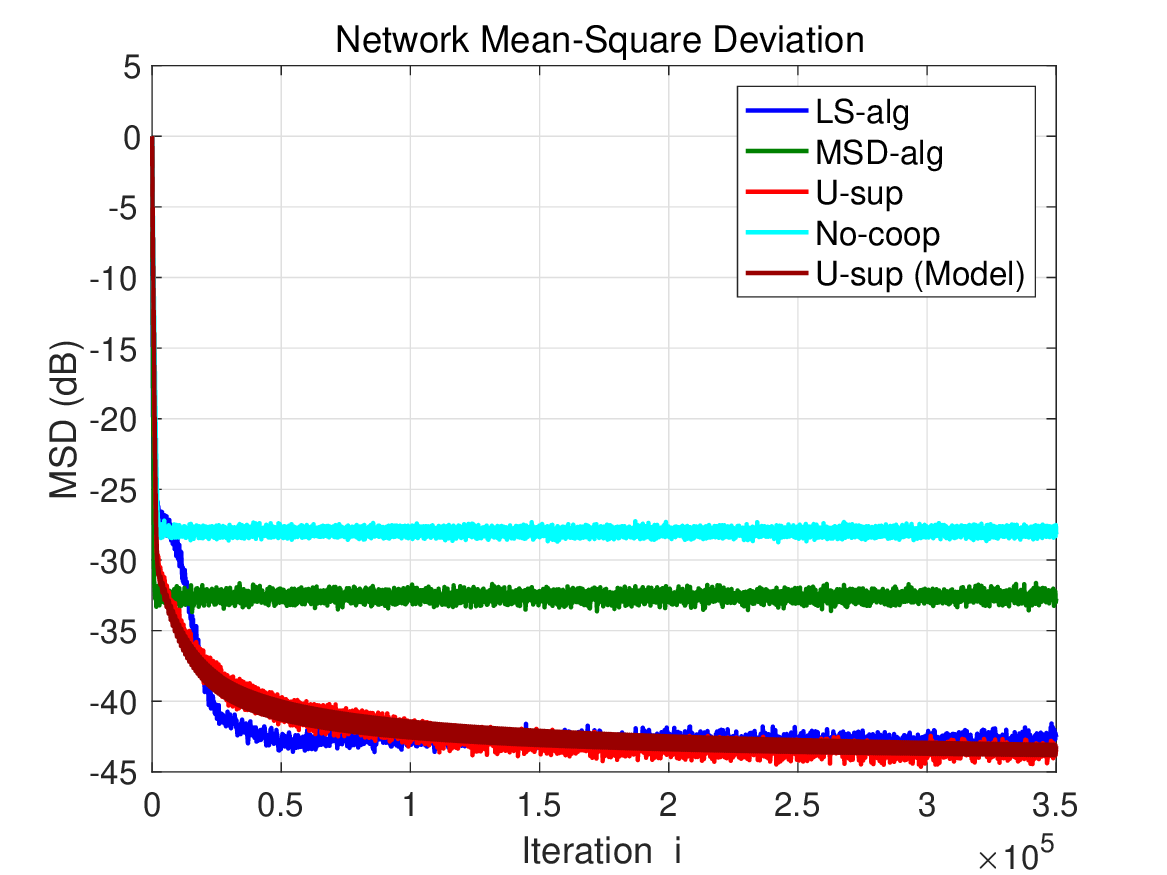}
		\small{(a)}
	\end{minipage}
	\hfill
	\begin{minipage}[c]{0.49\columnwidth}
	    \centering
		\includegraphics[width=\columnwidth]{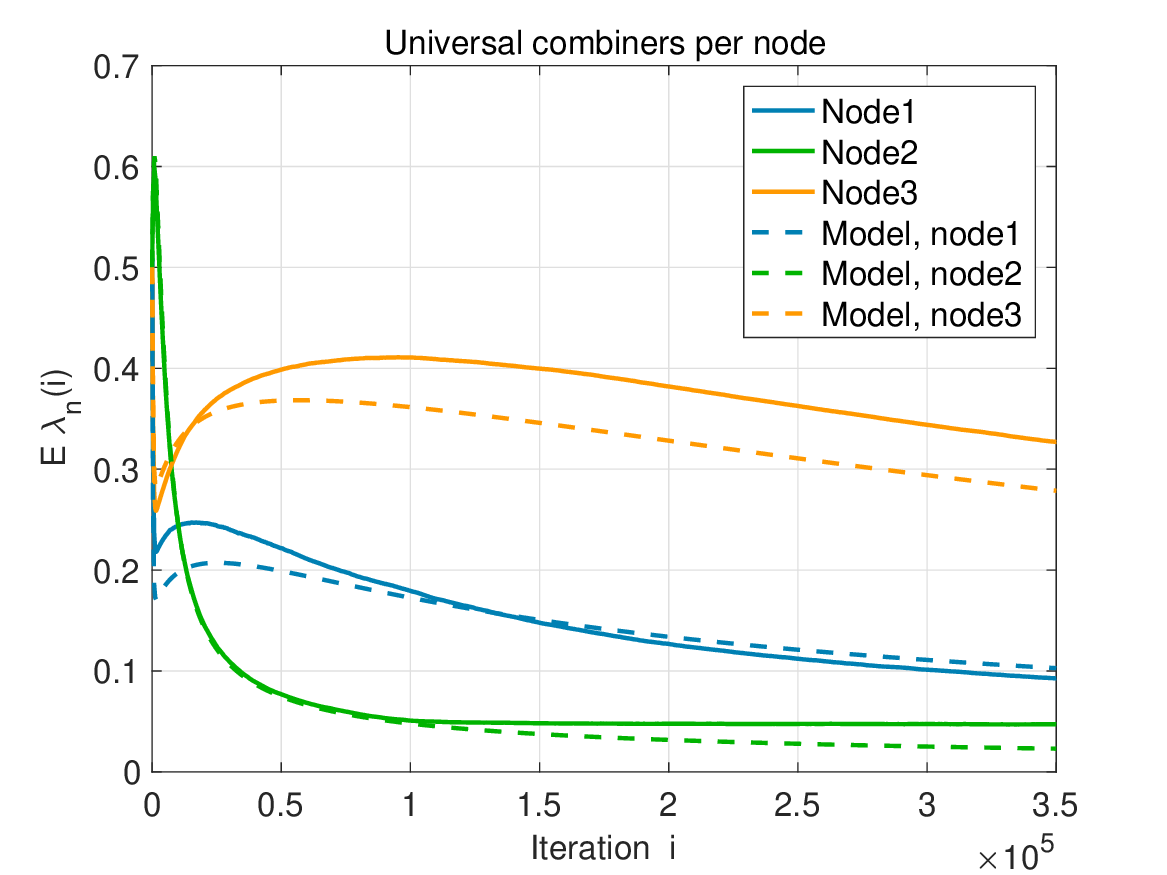}
		\small{(b)}
	\end{minipage}
%	\vspace*{0.2cm}
	\caption{\textbf{Example 5 (White inputs)}: (a) Network MSD$(i)$ and the theoretical model for U-sup (dark red); (b) The mean adaptive combiners E $\lambda_{n}(i)$ for a few nodes and their theoretical model (dashed lines), all corresponding to Fig. \ref{fig:Example5_MSD_eta}-(a).}
		\label{fig:Example5_MSD_eta}
	\vspace{-0.3cm}
\end{figure}

\add{The last example, \textbf{Example 6}, shows the effect of $L$ in a network with $N=8$ and $M=6$ as in Example 5.  Figure \ref{fig:msdmax} shows the worst node MSD at various  iterations during convergence, for different values of $L$.  One can see how intermediate values of $L$ (between $10$ and $1,000$) accelerate convergence, without affecting the steady-state performance. In other words, the U-sup algorithm is relatively insensitive to the choice of $L$ over a wide range.
\begin{figure}[htbp]
\centering\includegraphics[width = 0.45\textwidth]{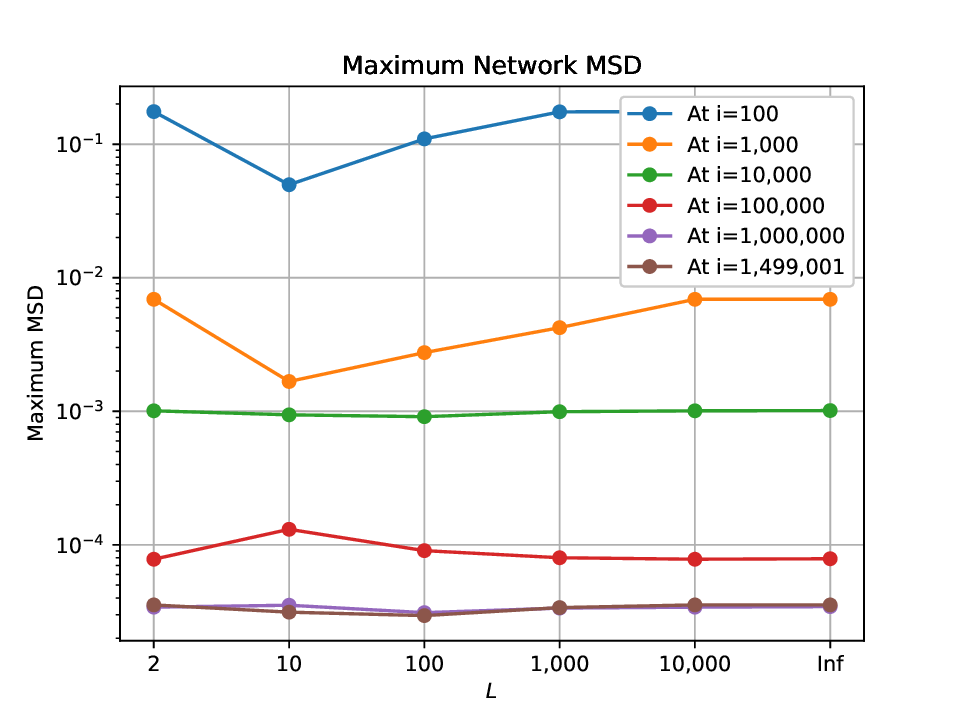}
\caption{\add{Performance comparison with different values of $L$ for a $8$-node network: worst-case node MSD at various points during convergence.  The values were obtained by averaging the MSD values at each node separately between instants $i$ and $i+999$, and taking the maximum value between the nodes.}}\label{fig:msdmax}
\end{figure}
}

\section{Conclusion}\label{sec:Conclusion}

This work has established the concept of universal estimation in the context of distributed adaptive estimation, also proposing a distributed adaptive protocol capable of attaining \add{distributed} global universality, as proved by Theorem \ref{thm:Global_universal} and shown by several simulations: the distributed supervisor drives the entire network to the best node performance, also efficiently rejecting poorly performing nodes, while promoting performance uniformity across the nodes.

Fair comparisons were carried out with two other directly competing algorithms, namely the MSD-alg (\ref{E:MSDSupervisor}) \cite{Takahashi10d}, and the LS-alg (\ref{E:LSAdaptiveSupervisor}) \cite{Bes2017}. The proposed U-sup algorithm (\ref{E:UniversalSupervisor}) is the simplest and was the only method to consistently achieve universality in all tested scenarios, under white and correlated data, for stationary and fast-varying plants. Algorithm (\ref{E:LSAdaptiveSupervisor}) performs very well for some stationary and slowly varying plants, however its computational complexity may limit its use in some applications.

Theoretical mean and mean-square error models were developed with a reasonable agreement with simulations, capturing the general tendencies of network MSD and local combiners $\text{E}\lambda_n(i)$, and proving convergence of the algorithm in the mean. The agreement between simulated and analytical curves improves as the local AF step-size decreases ($\mu_n\to0$); in the distributed case the same effect is further noticed when the universal combiner stepsize is also decreased, i.e., $\mu_a\to0$.

\ifCLASSOPTIONcaptionsoff
  \newpage
\fi

%-----------------------

% \reftitle{References}

% %=====================================
% % References, variant A: external bibliography
% %=====================================

%\externalbibliography{yes}

% %\bibliography{your_external_BibTeX_file}

\bibliographystyle{IEEEtran}

\bibliography{IEEEabrv,af,adanet,afcomb,math,sp,telecom,consensus_bibliography,SelectiveCooperation}
% %\end{paracol}
%-----------------------



% \end{thebibliography}


\begin{IEEEbiography}
[{\includegraphics[width=1in,height=1.25in,clip,keepaspectratio]{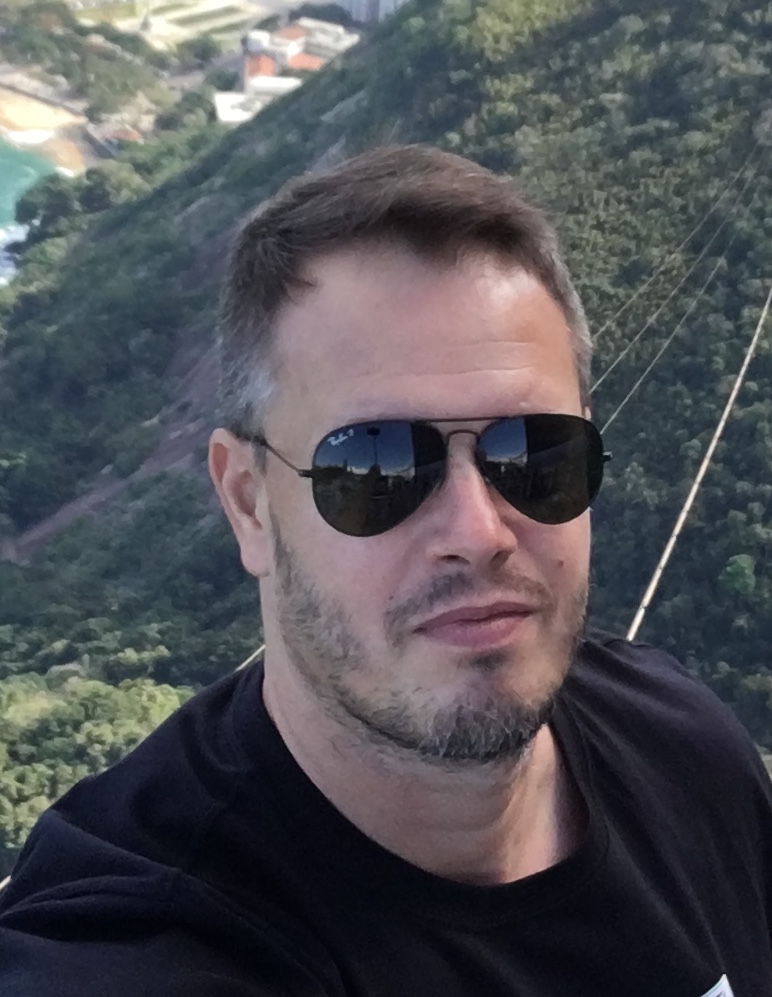}}]{Cássio G.\ Lopes} (S’06–M’08–SM’16) received his B.S. degree and M.S. degrees in electrical engineering from the Federal University of Santa Catarina, Brazil, in 1997 and 1999, respectively, and M.S. and Ph.D. degrees in electrical engineering from the University of California, Los Angeles (UCLA) in 2004 and 2008, respectively. From 2005 to 2007 he worked in a joint UCLA/NASA Jet Propulsion Laboratory project to develop frequency tracking algorithms to support direct-to-Earth Mars communications during entry, descent, and landing. In 2008 he worked at the Instituto Tecnologico de Aeronautica (ITA) Sao Jose dos Campos, Brazil as a postdoctoral researcher, developing distributed estimation algorithms for inertial navigation. In 2008 he joined the Department of Electronic Systems of the University of Sao Paulo (USP), Polytechnic School, where he is an associate professor of electrical engineering since 2014. From 2014 to 2018 he worked in a joint USP/EMBRAER project to enhance acoustic emission techniques for Structural Health Monitoring (SHM) of aircrafts. His current research interests are theory and methods for adaptive and statistical signal processing, distributed adaptive estimation, geometric algebra and tensor adaptive processing, as well as SHM.
\end{IEEEbiography}

\begin{IEEEbiography}
[{\includegraphics[width=1in,height=1.25in,clip,keepaspectratio]{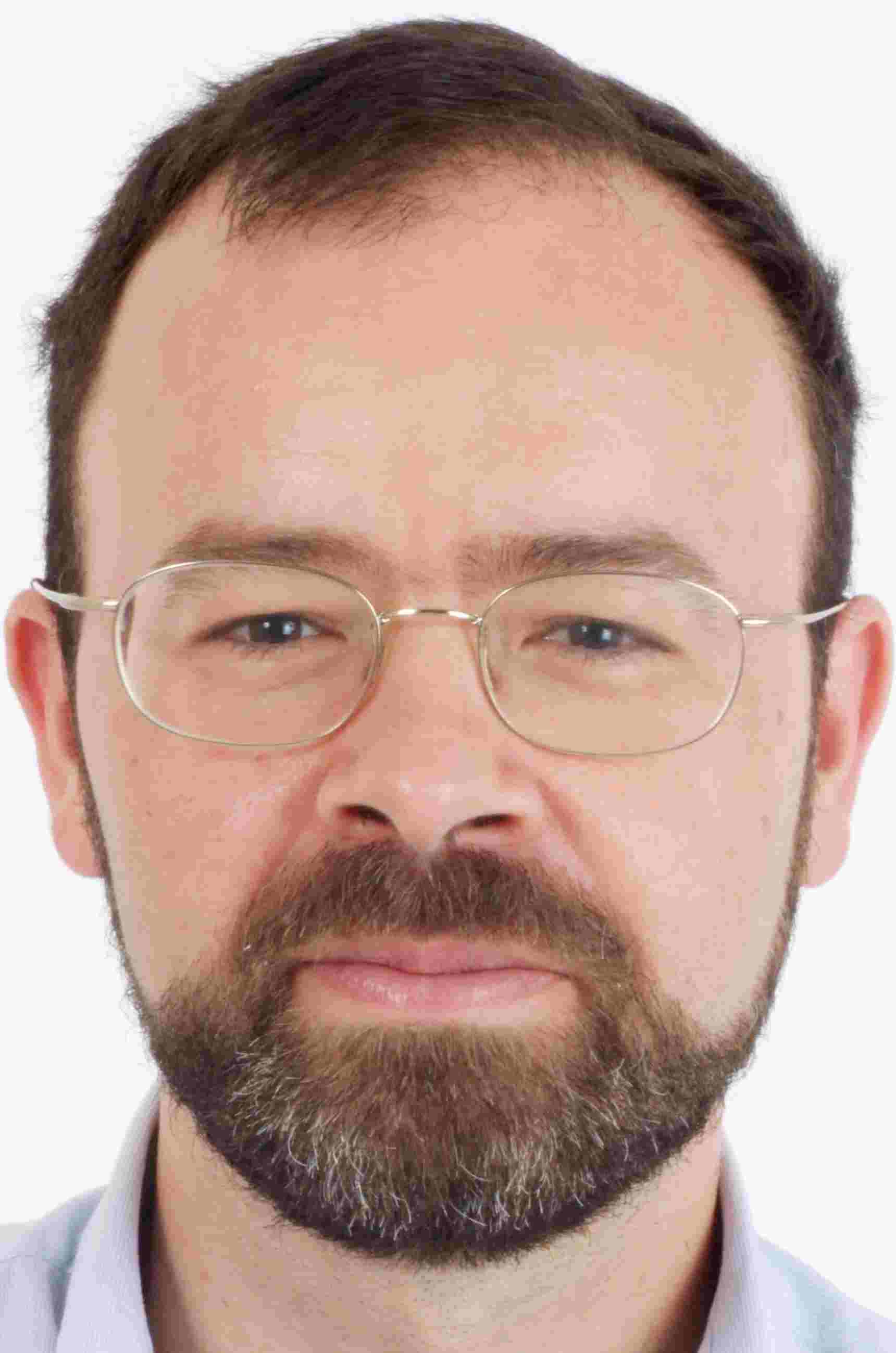}}]{Vítor H.\ Nascimento} obtained the B.S. and M.S. degrees in Electrical Engineering from Escola Politécnica, University of São Paulo, Brazil, in 1989 and 1992, respectively, and the Ph.D. degree from the University of California, Los Angeles, in 1999.  From 1990 to 1994 he was a Lecturer at the Univ. of São Paulo, and in 1999 he joined the faculty at the same school where his now Professor and chair of the Dept. of Electronic Systems Engineering.  One of his papers received the 2002 IEEE SPS Best Paper Award. He served as an Associate Editor for the IEEE Signal Processing Letters from 2003 to 2005, for the IEEE Transactions on Signal Processing from 2005 to 2008 and for the EURASIP Journal on Advances in Signal Processing from 2006 to 2009, and as a Senior Area Editor for the IEEE Trans. on Signal Processing (2018-2021).  He was a member of the IEEE-SPS Signal Processing Theory and Methods Technical Committee (2007 -2012 and 2016-2021). From 2010 to 2014 he was chair of the São Paulo IEEE-SPS Chapter, and between 2012 and 2016 he served as area editor for the Journal of Communication and Information Systems. He was Technical Chair of the 2014 International Telecommunications Symposium, organized in São Paulo by the Brazilian Telecommunications Society (SBrT), of the 2016 IEEE Sensor Array and Multichannel Signal Processing Workshop (Rio de Janeiro, Brazil), and of the 2021 IEEE Statistical Signal Processing Workshop (Rio de Janeiro). His research interests include signal processing theory and applications, statistical methods for epidemiology, adaptive and sparse estimation, distributed learning, structural health monitoring, array signal processing, and applied linear algebra.
\end{IEEEbiography}

\begin{IEEEbiography}
[{\includegraphics[width=1in,height=1.25in,clip,keepaspectratio]{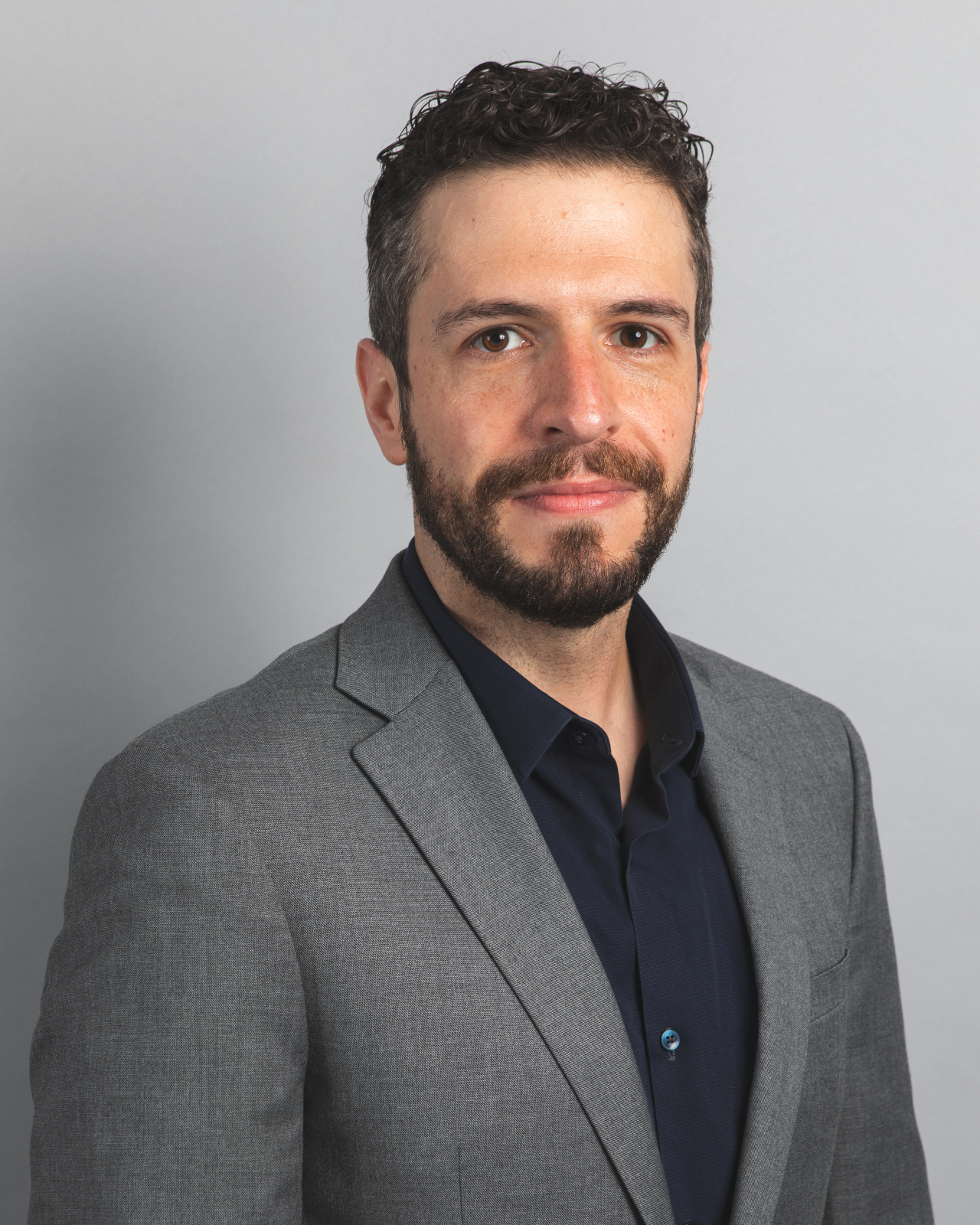}}]{Luiz F.\ O.\ Chamon} received the B.Sc. and M.Sc. degrees in electrical engineering from the University of São Paulo, São Paulo, Brazil, in 2011 and 2015 and the Ph.D. degree in electrical and systems engineering from the University of Pennsylvania~(Penn), Philadelphia, in 2020. Until 2022, he was a postdoctoral fellow at the Simons Institute of the University of California, Berkeley. He is currently an independent research group leader at the University of Stuttgart, Germany. In 2009, he was an undergraduate exchange student of the Masters in Acoustics of the École Centrale de Lyon, Lyon, France, and worked as an Assistant Instructor and Consultant on nondestructive testing at INSACAST Formation Continue. From 2010 to 2014, he worked as a Signal Processing and Statistics Consultant on a research project with EMBRAER. He received both the best student paper and the best paper awards at IEEE ICASSP 2020 and was recognized by the IEEE Signal Processing Society for his distinguished work for the editorial board of the IEEE Transactions on Signal Processing in 2018. His research interests include optimization, signal processing, machine learning, statistics, and control.
\end{IEEEbiography}


% %\newpage



%\vfill


\end{document}